\documentclass[11pt, notitlepage]{article}
\usepackage{amssymb,amsmath,comment}
\usepackage{amsthm}
\catcode`\@=11 \@addtoreset{equation}{section}
\def\thesection{\arabic{section}}

\def\theequation{\thesection.\arabic{equation}}
\catcode`\@=12
\usepackage{colortbl}
\usepackage{hyperref}
\usepackage[mathscr]{eucal}
\usepackage{epsf}
\usepackage{esint}
\usepackage{a4wide}
\def\R{\mathbb{R}}

\newcommand{\e}{\varepsilon}

\newcommand{\Om} {\Omega}

\newcommand{\noi} {\noindent}

\newcommand{\Tail} {\mathrm{Tail}}
\newcommand{\loc} {\mathrm{loc}}

\usepackage[all]{xy}
\catcode`\@=11
\def\theequation{\@arabic{\c@section}.\@arabic{\c@equation}}
\catcode`\@=12

\newtheorem{Theorem}{Theorem}[section]
\newtheorem{Lemma}[Theorem]{Lemma}
\newtheorem{prop}[Theorem]{Proposition}

\newtheorem{Remark}[Theorem]{Remark}
\newtheorem{Definition}[Theorem]{Definition}

\newtheorem*{ack}{Acknowledgements}

\newcommand{\nc}{\normalcolor}
\begin{document}

{\vspace{0.01in}}

\title{Higher H\"older regularity for the fractional $p$-Laplace equation in the subquadratic case}

\author{Prashanta Garain and Erik Lindgren}

\maketitle

\begin{abstract}\noindent
We study the fractional $p$-Laplace equation
$$
(-\Delta_p)^s u = 0
$$ for $0<s<1$ and in the subquadratic case $1<p<2$. We provide H\"older estimates with an explicit H\"older exponent. The inhomogeneous equation is also treated and there the exponent obtained is almost sharp for a certain range of parameters\nc. Our results complement the previous results for the superquadratic case when $p\geq 2$. The arguments are based on a careful Moser-type iteration and a perturbation argument.
\end{abstract}

\maketitle

\noi {Keywords: Fractional $p$-Laplacian, nonlocal elliptic equations, singular elliptic equations, H\"older regularity, Moser iteration.}

\noi{\textit{2020 Mathematics Subject Classification: 35B65, 35J75, 35R09.}

\bigskip

\tableofcontents

\section{Introduction} We study the local regularity of solutions of the nonlinear and nonlocal  equation
\begin{equation}\label{maineqn}
(-\Delta_p)^{s}u=f,
\quad 
\end{equation}
where
\begin{equation*}\label{fracplap}
(-\Delta_p)^{s}u(x)=\text{P.V.}\int_{\mathbb{R}^N}\frac{|u(x)-u(y)|^{p-2}(u(x)-u(y))}{|x-y|^{N+ps}}\,dy,
\end{equation*}
is the fractional $p$-Laplace operator, where P.V. denotes the principal value. It arises as the first variation of the Sobolev-Slobodecki\u{\i} seminorm for $W^{s,p}(\mathbb{R}^N)$, that is, as the first variation of the functional
\[
u\mapsto \iint_{\mathbb{R}^N\times \mathbb{R}^N} \frac{|u(x)-u(x)|^p}{|x-y|^{N+s\,p}}\,dx\,dy.
\]
This operator has generated vast activities in recent years. The main contribution of our work is to provide H\"older regularity of weak solutions of equation \eqref{maineqn}, with an explicit H\"older exponent. This is done in  Theorem \ref{teo:1} and Theorem \ref{teo:2}. Our results complement the existing results for the superquadratic case, $p\geq 2$, obtained in \cite{BLS}. To the best of our knowledge, this is the first result with an explicit H\"older exponent in the subquadratic case, $1<p<2,$ even for the homogeneous equation. We seize the moment to mention that we can verify that whenever
$$
1-\frac{1}{p}\geq s\geq \frac{N}{q},
$$
the H\"older exponent $\Theta$ obtained in the inhomogeneous setting for $f\in L^q$, that is
\[
{\Theta}=\frac{1}{p-1}\left(s\,p-\frac{N}{q}\right),
\]
{\it is the sharp one}, see Section \ref{sec:sharp} below. Hence, {\it our result is sharp} under this assumptions.\nc

\subsection{Main results}\label{sec:main}
We now present the main results of the paper. For the details regarding the notation used in the two theorems below, such as $\mathrm{Tail}_{p-1,s\,p}$, we refer to Section \ref{sec:prel}.
\begin{Theorem}[Almost $sp/(p-1)$-regularity]
\label{teo:1}
Let $\Omega\subset\mathbb{R}^N$ be a bounded and open set and assume that $1<p<2$ and $0<s<1$. Suppose $u\in W^{s,p}_{\rm loc}(\Omega)\cap L^{p-1}_{s\,p}(\mathbb{R}^N)$ is a local weak solution of 
\[
(-\Delta_p)^s u=0\qquad \mbox{ in }\Omega.
\] 
Then $u\in C^{\Gamma-\e}_{\rm loc}(\Omega)$ for every $\e\in(0,\Gamma)$, where
$$
\Gamma = \min(sp/(p-1),1).
$$
\par
In particular, for every $\e\in(0,\Gamma)$ and for every ball $B_{2R}(x_0)\Subset\Omega$, there exist constants $\sigma=\sigma(N,s,p,\e)\in(0,1)$ and $C=C(N,s,p,\e)>0$ such that
\begin{equation*}
\label{apriori}
[u]_{C^{\Gamma-\e}(B_{\sigma R}(x_0))}\leq \frac{C}{R^{\Gamma-\e}}\,\left(\|u\|_{L^\infty(B_{R}(x_0))}+\mathrm{Tail}_{p-1,s\,p}(u;x_0,R)\right).
\end{equation*}
\end{Theorem}
\begin{Theorem}
\label{teo:2}
Let $\Omega\subset\mathbb{R}^N$ be a bounded and open set and assume that $1<p<2$ and $0<s<1$. Suppose $u\in W^{s,p}_{\rm loc}(\Omega)\cap L^{p-1}_{s\,p}(\mathbb{R}^N)$ is a local weak solution of 
\[
(-\Delta_p)^s u=f\qquad \mbox{ in }\Omega,
\] 
where $f\in L^q_\text{loc}(\Omega)$ with 
\[
\left\{\begin{array}{lr}
q>\frac{N}{sp},& \mbox{ if } s\,p\leq N,\\
q\ge 1,& \mbox{ if }s\,p>N.
\end{array}
\right.
\]
 Let $$
\Theta = \min\Big(1,\frac{sp-N/q}{p-1}\Big).
$$
Then $u\in C^{\Theta-\e}_{\rm loc}(\Omega)$ for every $\e\in(0,\Theta)$.
\par
In particular, for every $\e\in(0,\Theta)$ and for every ball $B_{4R}(x_0)\Subset\Omega$, there exists a constant $C=C(N,s,p,q,\e)>0$ such that
\begin{equation*}
\label{apriori1}
[u]_{C^{\Theta-\e}(B_{ R/8}(x_0))}\leq \frac{C}{R^{\Theta-\e}}\,\left(\|u\|_{L^\infty(B_{2R}(x_0))}+\mathrm{Tail}_{p-1,s\,p}(u;x_0,2R)+\left(R^{sp-N/q}\|f\|_{L^{q}(B_{2R}(x_0))}\right)^\frac{1}{p-1}\right).
\end{equation*}
\end{Theorem}

\begin{Remark} It is worth mentioning that in the case when $q<\infty$ and $\Theta =(sp-N/q)/(p-1)$, a careful inspection of the proof of Theorem \ref{teo:2} reveals that we obtain the stronger result that $u\in C^{\Theta}_{\rm loc}(\Omega)$, with a similar estimate.
\end{Remark}
\subsection{Comments on the results}\label{sec:sharp} 
We now discuss the sharpness of our results, in particular Theorem \ref{teo:2}. Choose $N,p,q$ such that
$$
1-\frac{1}{p}\geq \frac{N}{q}
$$
and pick $s\in [N/q,1-1/p]$. Then it follows that $sp <N$, $q >N/sp$, $sp\leq p-1$ and that $(sp-N/q)/(p-1)\geq s$.

Define for some $\e>0$ the function
$$u(x)=|x|^{\gamma+\e}, \quad \gamma = (sp-N/q)/(p-1).$$
By the assumptions, $u\in C^{s+\e}_{\text{loc}}(\R^N)\cap W^{s,p}_{\text{loc}}(\R^N)\cap L_{sp}^{p-1}(\R^N)$. In addition, by homogeneity and radial symmetry it follows that
$$
(-\Delta_p)^s u = f, \quad f(x)=C(s,p,\gamma,\e)|x|^{(\gamma+\e-s)(p-1)-s}
$$
where $f\in L^q_{\text{loc}}(\R^N)$ if and only if $\e>0$. It is clear that $u\not\in C^{\alpha}(B_1)$ for any $\alpha>\gamma+\e$. Therefore, the result is sharp in this region of parameters.
 \nc
Now we comment on the assumptions on $q$ and $p$ in Theorem \ref{teo:2}. We believe that they are sharp and they do perfectly match the sharp assumptions in the local limit. Indeed, in the local case, that should correspond to the limiting case $s=1$, the assumptions {become} $q>N/p$ when $p\leq N$ and $q\geq 1$ when $p>N$. These are the proper conditions for the inhomogeneous $p$-Laplace equation, see \cite{Tei13} and \cite{TU14}. \nc

\subsection{Known results}
The first appearance of equations similar to the fractional $p$-Laplacian that we are aware of is in \cite{IN}.\nc There existence, uniqueness, and the convergence to the $p-$Laplace equation as $s$ goes to $1$, are proved in the viscosity setting. The starting point of the regularity theory was \cite{DKP}, where the local H\"older regularity was proved, using a nonlocal De Giorgi-type method. See also \cite{DKP2}, for a related Harnack inequality. The paper \cite{BP} contains several useful regularity estimates for the inhomogeneous equation.

The literature on related H\"older regularity results is vast and we only mention a fraction. A local regularity result using viscosity methods was obtained in \cite{Li}.  In \cite{Co}, nonlocal analogues of the De Giorgi classes are introduced and used to prove regularity results in a very general setting. We also seize the opportunity to mention that fractional De Giorgi classes has been used in the context of local equations in \cite{Mi}.

The regularity up to the boundary has been studied in \cite{IMS} and \cite{IMS2}. Basic H\"older regularity up to the boundary is proved for general $p$ and for $p\geq 2$ finer regularity results up to the boundary are obtained.

 In terms of regularity for the inhomogeneous equation, we mention the papers  \cite{BLS}, \cite{DS23} and \cite{KMS}. In \cite{KMS}, the authors study the regularity for equations of the type \eqref{maineqn} with a right hand side $f$ belonging to a Lorentz space. Sharp results for when $u$ is continuous are obtained. The paper \cite{BLS} is the counterpart of the present paper in the superquadratic case. In \cite{DS23}, these results are improved and the authors obtain sharp H\"older regularity results when $p\geq 2$ and when the right hand side belongs to a Marcinkiewicz space.

We stress that for the subquadratic case $1<p<2$, none of the above mentioned papers include an explicit H\"older exponent.

In addition to H\"older regularity, there has been quite some development of results in terms of higher Sobolev differentiability. In the linear case $p=2$, see   \cite{ABES19},  \cite{BWZ17}, \cite{Coz15},  \cite{KMSapde} and \cite{KNS22}, where the results are valid for more general kernels. For a general $p$, this has been studied in    \cite{Brolin}, \cite{BK23}, \cite{DS23} and \cite{Sch16}.

We finally mention that the corresponding results for the $p$-Laplacian are well known. See for instance \cite{Tei13} and \cite{TU14}.
\subsection{Plan of the paper}

In Section \ref{sec:prel}, we discuss notation, definitions and certain results in function spaces. The most important part of the paper is Section \ref{sec:hom}, where we prove Theorem \ref{teo:1}, using a Moser-type argument that results in an improved differentiability that can be iterated in an unusual way. Following this, in Section \ref{sec:inhomo}, we treat the inhomogeneous equation, by means of a perturbation argument, using the regularity obtained for the homogeneous equation. Finally, in the Appendix, we include a list of pointwise inequalities that are used throughout the paper.

\begin{ack}
We are grateful to Lorenzo Brasco for many fruitful discussions related to this manuscript and in particular for his useful suggestions regarding the proof of Theorem \ref{teo:holderf}. We also thank Alireza Tavakoli for his suggestions.

E. L. has been supported by the Swedish Research Council, grant no. 2023-03471, 2017-03736 and 2016-03639 under the development of this project. Part of this material is based upon work supported by the Swedish Research Council under grant no. 2016-06596 while the second author were participating in the research program ``Geometric Aspects of Nonlinear Partial Differential Equations'', at Institut Mittag-Leffler in Djursholm, Sweden, during the fall of 2022.
 \end{ack}
\bigskip

\noindent {\bf Conflict of interest statement:} On behalf of all authors, the corresponding author states that there is no conflict of interest. 

\bigskip

\noindent {\bf Data availability statement: } The manuscript has no associated data.

\section{Preliminaries}\label{sec:prel}

In this section we present some auxiliary results needed in the rest of the paper.
\subsection{Notation}
Throughout the paper, we shall use the following notation: $B_r(x_0)$ denotes the ball of radius $r$ with center at $x_0$. When $x_0=0$, we write $B_r(0):=B_r$. For a function $u$, we denote the positive and the negative part of $u$ as $u_{\pm}=\max\{\pm u,0\}$. The conjugate exponent $\frac{l}{l-1}$ of $l>1$ will be denoted by $l'$. We write
$c$ or $C$ to denote a positive constant which may vary from line to line or even in
the same line. The dependencies on parameters are written in the parentheses.

For $1<q<\infty$, we define the function $J_q:\mathbb{R}\to\mathbb{R}$ by
\begin{equation}\label{jp}
J_q(t)=|t|^{q-2}t
\end{equation}
and for $0<s<1$ and $1<p<\infty$ we use the notation
\begin{equation}\label{dmu}
d\mu=\frac{dx dy}{|x-y|^{N+ps}}.
\end{equation}
Moreover, for {$0\leq \delta\leq 1$, we use the notation
$$
[u]_{C^\delta(\Omega)}:=\sup_{x\neq y\in \Omega}\frac{|u(x)-u(y)|}{|x-y|^{\delta}}.
$$
We also define
\[
\psi_h(x)=\psi(x+h),\qquad \delta_h \psi(x)=\psi_h(x)-\psi(x)\\
\]
and
\[
 \delta^2_h \psi(x)=\delta_h(\delta_h \psi(x))=\psi_{2\,h}(x)+\psi(x)-2\,\psi_h(x)
\]
for functions $\psi:\mathbb{R}^N\to\mathbb{R}$ and $h\in\mathbb{R}^N$. Note that the following discrete product rule holds:
\begin{equation}\label{Lrule1}
\delta_h(\phi\psi)=\psi_h\delta_h\phi+\phi\delta_h\psi.
\end{equation}
\subsection{Function spaces}
It will be necessary to introduce two Besov-type spaces. For this reason, let $1\le q<\infty$ and $\psi\in L^q(\mathbb{R}^N)$. For $0<\beta\le 1$, define
\[
[\psi]_{\mathcal{N}^{\beta,q}_\infty(\mathbb{R}^N)}:=\sup_{|h|>0} \left\|\frac{\delta_h \psi}{|h|^{\beta}}\right\|_{L^q(\mathbb{R}^N)},
\]
and for $0<\beta<2$, define
\[
[\psi]_{\mathcal{B}^{\beta,q}_\infty(\mathbb{R}^N)}:=\sup_{|h|>0} \left\|\frac{\delta_h^2 \psi}{|h|^{\beta}}\right\|_{L^q(\mathbb{R}^N)}.
\]
The Besov-type spaces $\mathcal{N}^{\beta,q}_\infty$ and $\mathcal{B}^{\beta,q}_\infty$ are defined by 
\[
\mathcal{N}^{\beta,q}_\infty(\mathbb{R}^N)=\left\{\psi\in L^q(\mathbb{R}^N)\, :\, [\psi]_{\mathcal{N}^{\beta,q}_\infty(\mathbb{R}^N)}<+\infty\right\},\qquad 0<\beta\le 1,
\]
and
\[
\mathcal{B}^{\beta,q}_\infty(\mathbb{R}^N)=\left\{\psi\in L^q(\mathbb{R}^N)\, :\, [\psi]_{\mathcal{B}^{\beta,q}_\infty(\mathbb{R}^N)}<+\infty\right\},\qquad 0<\beta<2.
\]
The {\it Sobolev-Slobodecki\u{\i} space} is defined as
\[
W^{\beta,q}(\mathbb{R}^N)=\left\{\psi\in L^q(\mathbb{R}^N)\, :\, [\psi]_{W^{\beta,q}(\mathbb{R}^N)}<+\infty\right\},\qquad 0<\beta<1,
\]
where the seminorm $[\,\cdot\,]_{W^{\beta,q}(\mathbb{R}^N)}$ is given by
\[
[\psi]_{W^{\beta,q}(\mathbb{R}^N)}=\left(\iint_{\mathbb{R}^N\times \mathbb{R}^N} \frac{|\psi(x)-\psi(y)|^q}{|x-y|^{N+\beta\,q}}\,dx\,dy\right)^\frac{1}{q}.
\]
The above spaces are endowed with the norms
\[
\|\psi\|_{\mathcal{N}^{\beta,q}_\infty(\mathbb{R}^N)}=\|\psi\|_{L^q(\mathbb{R}^N)}+[\psi]_{\mathcal{N}^{\beta,q}_\infty(\mathbb{R}^N)},
\] 
\[
\|\psi\|_{\mathcal{B}^{\beta,q}_\infty(\mathbb{R}^N)}=\|\psi\|_{L^q(\mathbb{R}^N)}+[\psi]_{\mathcal{B}^{\beta,q}_\infty(\mathbb{R}^N)},
\]
and
\[
\|\psi\|_{W^{\beta,q}(\mathbb{R}^N)}=\|\psi\|_{L^q(\mathbb{R}^N)}+[\psi]_{W^{\beta,q}(\mathbb{R}^N)}.
\]
We also introduce the space $W^{\beta,q}(\Omega)$ for a subset $\Omega\subset \mathbb{R}^N$,
\[
W^{\beta,q}(\Omega)=\left\{\psi\in L^q(\Omega)\, :\, [\psi]_{W^{\beta,q}(\Omega)}<+\infty\right\},\qquad 0<\beta<1,
\]
where naturally
\[
 [\psi]_{W^{\beta,q}(\Omega)}=\left(\iint_{\Omega\times \Omega} \frac{|\psi(x)-\psi(y)|^q}{|x-y|^{N+\beta\,q}}\,dx\,dy\right)^\frac{1}{q}.
\]
It will also be convenient to use the following abuse of notation for the Sobolev exponent $p_s^*$ related to the space $W^{s,p}$: if $sp<N$ then 
$$
p_s^*= \frac{Np}{N-sp}, \quad (p_s^*)' = \frac{Np}{Np-N+sp}
$$
and if  $sp>N$ then 
$$
p_s^* =\infty ,\quad (p_s^*)'=1.
$$

\subsection{Embedding inequalities} The following result can be found for example in \cite[Lemma 2.3]{Brascoccm}.
\begin{Lemma}\label{emb1}
The following embedding 
$$
\mathcal{B}_\infty^{\beta,q}(\mathbb{R}^N)\hookrightarrow \mathcal{N}_\infty^{\beta,q}(\mathbb{R}^N)
$$
is continuous, provided $0<\beta<1$ and $1\leq q<\infty$. Moreover,
$$
[\psi]_{\mathcal{N}_\infty^{\beta,q}(\mathbb{R}^N)}\leq\frac{C}{1-\beta}[\psi]_{\mathcal{B}_\infty^{\beta,q}(\mathbb{R}^N)},
$$
for every $\psi\in \mathcal{B}_\infty^{\beta,q}(\mathbb{R}^N)$,
for some constant $C=C(N,q)>0$.
\end{Lemma}
We have the following embedding result from \cite[Theorem 2.8]{BLS}.
\begin{Theorem}\label{emb2}
Let $\psi\in \mathcal{N}_\infty^{\beta,q}(\mathbb{R}^N)$, where $0<\beta<1$ and $1\leq q<\infty$ such that $\beta q>N$. Then for every $0<\alpha<\beta-\frac{N}{q}$, we have $\psi\in C^\alpha_{\mathrm{loc}}(\mathbb{R}^N)$. More precisely,
$$
\sup_{x\neq y}\frac{|\psi(x)-\psi(y)|}{|x-y|^\alpha}\leq C\Big([\psi]_{\mathcal{N}_\infty^{\beta,q}(\mathbb{R}^N)}\Big)^\frac{\alpha q+N}{\beta q}\Big(\|\psi\|_{L^q(\mathbb{R}^N)}\Big)^\frac{(\beta-\alpha)q-N}{\beta q},
$$
for some positive constant $C=C(N,q,\alpha,\beta)$ which blows up as $\alpha\nearrow \beta-\frac{N}{q}$.
\end{Theorem}

The following result follows from \cite[Proposition 2.7]{BLS}.

\begin{Lemma}\label{BLSprop}
Let $\Omega\subset\mathbb{R}^N$ be an open and bounded set. Assume that $1<p<\infty$ and $0<s<1$. Then 
$$
\|u\|^p_{L^p(\Omega)}\leq C|\Om|^\frac{sp}{N}\int_{\mathbb{R}^N}\int_{\mathbb{R}^N}\frac{|u(x)-u(y)|^p}{|x-y|^{N+ps}}\,dx dy,
$$
holds for every $u\in W^{s,p}(\mathbb{R}^N)$ such that $u=0$ almost everywhere
in $\mathbb{R}^N\setminus\Omega$, for some positive constant $C=C(N,p,s)$.
\end{Lemma}

\subsection{Tail spaces and weak solutions} For a priori estimates, the so-called {\it tail spaces} that takes into account the global behavior are expedient. The {\it tail space} is defined as
\[
L^{q}_{\alpha}(\mathbb{R}^N)=\left\{u\in L^{q}_{\rm loc}(\mathbb{R}^N)\, :\, \int_{\mathbb{R}^N} \frac{|u|^q}{1+|x|^{N+\alpha}}\,dx<+\infty\right\},\qquad q>0 \mbox{ and } \alpha>0,
\]
and the global behavior of a function $u\in L^q_{\alpha}(\mathbb{R}^N)$ is measured by the quantity
\begin{equation*}\label{mtail}
\mathrm{Tail}_{q,\alpha}(u;x_0,R)=\left[R^\alpha\,\int_{\mathbb{R}^N\setminus B_R(x_0)} \frac{|u|^q}{|x-x_0|^{N+\alpha}}\,dx\right]^\frac{1}{q}.
\end{equation*}
Here $x_0\in\mathbb{R}^N$, $R>0,\,\beta>0$.

\begin{Definition}\label{subsupsolution}(Local weak solution)
Suppose $\Omega\subset\mathbb{R}^N$ is an open and bounded set. Assume that $1<p<2$ and $0<s<1$. Let $f\in L_\text{loc}^q(\Om)$ with $q\geq (p_s^{*})'$ if $sp\neq N$ and $q>1$ if $sp=N$. We define $u\in W_{\loc}^{s,p}(\Omega)\cap L^{p-1}_{sp}(\mathbb{R}^N)$ to be  a local weak solution of $(-\Delta_p)^s u=f$ in $\Omega$, if 
\begin{equation}\label{wksol}
\begin{gathered}
\int_{\mathbb{R}^N}\int_{\mathbb{R}^N}J_p(u(x)-u(y))(\phi(x)-\phi(y))\,d\mu=\int_\Omega f \phi \,dx,
\end{gathered}
\end{equation}
for every compactly supported $\phi\in W^{s,p}(\Omega)$, where $J_p(t)=|t|^{p-2}t$ and $d\mu=\frac{dx dy}{|x-y|^{N+ps}}$ are defined in \eqref{jp} and \eqref{dmu} respectively.
\end{Definition}

Now we define the notion of weak solution for the Dirichlet problem associated with $(-\Delta_p)^s$. To this end, for given open and bounded sets $\Om\Subset\Om'\subset\mathbb{R}^N$ and $g\in L_{sp}^{p-1}(\mathbb{R}^N)$, we define
$$
X_g^{s,p}(\Om,\Om'):=\{w\in W^{s,p}(\Om')\cap L_{p-1}^{sp}(\mathbb{R}^N):w=g\text{ a.e. in }\mathbb{R}^N\setminus\Om\}.
$$
\begin{Definition}\label{DP}(Dirichlet problem) Suppose $\Om\Subset\Om'\subset\mathbb{R}^N$ are two open and bounded sets. Assume that $1<p<2$ and $0<s<1$. Let $f\in L^q(\Om)$ with $q\geq (p_s^{*})'$ if $sp\neq N$ and $q>1$ if $sp=N$ and $g\in L^{p-1}_{sp}(\mathbb{R}^N)$. We define $u\in X_g^{s,p}(\Om,\Om')$ to be a weak solution of the boundary value problem
\begin{equation}\label{DPeqn}
(-\Delta_p)^s u=f\text{ in }\Om,\quad 
u=g\text{ a.e. in }\mathbb{R}^N\setminus\Om,
\end{equation}
if for every $\phi\in X_0^{s,p}(\Om,\Om')$, equation \eqref{wksol} holds.
\end{Definition} 
By Proposition 2.12 in \cite{BLS}, there {exists} a unique weak solution {of the Dirichlet problem \eqref{DPeqn}} in the sense above, given $g\in W^{s,p}(\Omega')\cap L_{sp}^{p-1}(\R^N)$.
\section{The homogeneous equation}\label{sec:hom}
In this section, we treat the regularity for the homogeneous equation. This is done through an iteration scheme built on improved Besov-type regularity and improved H\"older regularity.
\subsection{Improved Besov-type regularity}
The starting point is the following improved Besov-type regularity.
\begin{prop}
\label{prop:improve} 
Let $1<p <2$ and $0<s<1$. Assume that $u\in W^{s,p}_{\rm loc}(B_2)\cap L^{p-1}_{s\,p}(\mathbb{R}^N)$ is a local weak solution of $(-\Delta_p)^s u=0$ in $B_2$. Suppose that 
\begin{equation}
\label{bounds}
\|u\|_{L^\infty(B_1)}\leq 1, \qquad \mathrm{Tail}_{p-1,s\,p}(u;0,1)^{p-1}=\int_{\mathbb{R}^N\setminus B_1} \frac{|u(y)|^{p-1}}{|y|^{N+s\,p}}\,dy\leq 1 \qquad \mbox{ and }\quad [u]_{C^\gamma(B_1)}\,  \leq 1,
\end{equation} 
for some $\gamma\in [0,1)$. Moreover, suppose that for some $\alpha\in [0,1)$, {$1\leq q<\infty$ and $0<h_0<\frac{1}{10}$}, we have
\begin{equation}\label{asm1}
\sup_{0<|h|< h_0}\left\|\frac{\delta^2_h u }{|h|^{\alpha}}\right\|_{L^{q}(B_{R+4h_0})}
<+\infty.
\end{equation}
Then for $R$ such that $4\,h_0<R\le 1-5\,h_0$,  
we have
\begin{equation}\label{pro1}
\begin{split}
\sup_{0<|h|< h_0}\left\|\frac{\delta^2_h u}{|h|^{\frac{sp-\gamma(p-2)+\alpha q}{q+1}}}\right\|_{L^{q+1}(B_{R-4\,h_0})}^{q+1}\leq 
C\,\left(\sup_{0<|h|< h_0}\left\|\frac{\delta^2_h u }{|h|^\alpha}\right\|_{L^{q}(B_{R+4\,h_0})}^{q}+1\right),
\end{split}
\end{equation}
for some positive constant $C=C(N,s,p,q,h_0,\alpha,\gamma)$  and $C\nearrow +\infty$ as $h_0\searrow 0$.
\end{prop}
\begin{proof}
We divide the proof into five steps.\\
\textbf{Step 1: Discrete differentiation of the equation.} We set
$
r=R-4h_0$ and recall $\quad d\mu=\frac{dx dy}{|x-y|^{N+ps}}.$ Take $\varphi\in W^{s,p}(B_R)$ vanishing outside $B_{\frac{R+r}{2}}$. Let $h\in\mathbb{R}^N\setminus\{0\}$ be such that $|h|<h_0$. Testing \eqref{wksol} with $\varphi_{-h}$ and performing a change of variable yields
\begin{equation}
\label{differentiated}
\frac1h \int_{\mathbb{R}^N} \int_{\mathbb{R}^N} \Big(J_p(u_h(x)-u_h(y))-J_p(u(x)-u(y))\Big)\,\Big(\varphi(x)-\varphi(y)\Big)\,d\mu=0.
\end{equation}
In what follows, suppose $\eta\in C^\infty_0(B_R)$ is such that
\[
0\le \eta\le 1,\qquad \eta\equiv 0 \mbox{ in } \mathbb{R}^N\setminus B_{\frac{R+r}{2}},\qquad |\nabla \eta|\le \frac{C}{R-r}=\frac{C}{4\,h_0}.
\]
Testing \eqref{differentiated} with 
\[
\varphi=J_{q+1}\left(\frac{\delta_h u}{|h|^{\theta}}\right)\,\eta^2,
\]
we get
\begin{equation}\label{deqn1}
\begin{split}
\iint_{\mathbb{R}^N\times\mathbb{R}^N}& \frac{\Big(J_p(u_h(x)-u_h(y))-J_p(u(x)-u(y))\Big)}{|h|^{1+\theta\,q}}\\
&\times\Big(J_{q+1}(u_h(x)-u(x))\,\eta(x)^2-J_{q+1}(u_h(y)-u(y))\,\eta(y)^2\Big)\,d\mu=0.
\end{split}
\end{equation}
We split the above double integral into three pieces:
\[
\begin{split}
\mathcal{I}_1:=\iint_{B_R\times B_R}& \frac{\Big(J_p(u_h(x)-u_h(y))-J_p(u(x)-u(y))\Big)}{|h|^{1+\theta\,q}}\\
&\times\Big(J_{q+1}(u_h(x)-u(x))\,\eta(x)^2-J_{q+1}(u_h(y)-u(y))\,\eta(y)^2\Big)d\mu,
\end{split}
\]
\[
\begin{split}
\mathcal{I}_2:=\iint_{B_\frac{R+r}{2}\times (\mathbb{R}^N\setminus B_R)}& \frac{\Big(J_p(u_h(x)-u_h(y))-J_p(u(x)-u(y))\Big)}{|h|^{1+\theta\,q}}\,J_{q+1}(u_h(x)-u(x))\,\eta(x)^2\,d\mu,
\end{split}
\]
and
\[
\begin{split}
\mathcal{I}_3:=-\iint_{(\mathbb{R}^N\setminus B_R)\times B_\frac{R+r}{2}}& \frac{\Big(J_p(u_h(x)-u_h(y))-J_p(u(x)-u(y))\Big)}{|h|^{1+\theta\,q}}\,J_{q+1}(u_h(y)-u(y))\,\eta(y)^2\,d\mu,
\end{split}
\]
where we used that $\eta$ vanishes identically outside $B_{(R+r)/2}$. Thus the equation \eqref{deqn1} can be written as
\begin{equation}
\label{morsecode}
\mathcal{I}_1=-\mathcal{I}_2-\mathcal{I}_3.
\end{equation}
We estimate $\mathcal{I}_j$ for $j=1,2,3$ separately.
\vskip.2cm\noindent
{\bf Estimate of $\mathcal{I}_1$.}
We observe that
\[
\begin{split}
J_{q+1}(u_h(x)-u(x))\,\eta(x)^2&-J_{q+1}(u_h(y)-u(y))\,\eta(y)^2\\
&=\frac{\Big(J_{q+1}(u_h(x)-u(x))-J_{q+1}(u_h(y)-u(y))\Big)}{2}\,\Big(\eta(x)^2+\eta(y)^2\Big)\\
&+\frac{\Big(J_{q+1}(u_h(x)-u(x))+J_{q+1}(u_h(y)-u(y))\Big)}{2}\,(\eta(x)^2-\eta(y)^2).
\end{split}
\]
Therefore, we have
\begin{equation}\label{neqn0}
\begin{split}
{I}=\Big(J_p(u_h(x)-u_h(y))&-J_p(u(x)-u(y))\Big)\,\Big(J_{{q+1}}(u_h(x)-u(x))\,\eta(x)^2-J_{{q+1}}(u_h(y)-u(y))\,\eta(y)^2\Big)\\
&\geq \Big(J_p(u_h(x)-u_h(y))-J_p(u(x)-u(y))\Big)\\
&\times \Big(J_{q+1}(u_h(x)-u(x))-J_{q+1}(u_h(y)-u(y))\Big)\,\left(\frac{\eta(x)^2+\eta(y)^2}{2}\right)\\
&-\Big|J_p(u_h(x)-u_h(y))-J_p(u(x)-u(y))\Big|\\
&\times {\Big|J_{q+1}(u_h(x)-u(x))+J_{q+1}(u_h(y)-u(y))\Big|}\,\left|\frac{\eta(x)^2-\eta(y)^2}{2}\right|\\
&:=J_1-J_2.
\end{split}
\end{equation}
{{\bf Estimate of $J_1$:}} 
We will now estimate the positive term. With the notation
$$
a=u_h(x), \quad b=u_h(y),\quad c=u(x)\quad\text{ and } \quad d=u(y),
$$
we have by Lemma \ref{lemma:sing_ineq_2} together with the fact that $u$ is locally $\gamma$-H\"older continuous (recall {\eqref{bounds}})
\begin{equation}\label{neqn8}
\begin{split}
J_1&=\Big(J_p(a-b)-J_p(c-d)\Big)\Big(J_{q+1}(a-c)-J_{q+1}(b-d)\Big)\Big(\frac{\eta(x)^2+\eta(y)^2}{2}\Big)\\
&\geq C\left||a-c|^\frac{q-1}{2}(a-c)-|b-d|^\frac{q-1}{2}(b-d)\right|^2 \left(|a-b|^{p-2}+|c-d|^{p-2}\right)\Big(\frac{\eta(x)^2+\eta(y)^2}{2}\Big)\\
&\geq C\left||a-c|^\frac{q-1}{2}(a-c)-|b-d|^\frac{q-1}{2}(b-d)\right|^2|x-y|^{\gamma(p-2)}\Big(\frac{\eta(x)^2+\eta(y)^2}{2}\Big)\\
&=C\left||\delta_h u(x)|^\frac{q-1}{2}(\delta_h u(x))-|\delta_h u (y)|^\frac{q-1}{2}(\delta_h u(y))\right|^2(\eta(x)^2+\eta(y)^2) |x-y|^{\gamma(p-2)}, \quad C=C(p,q).
\end{split}
\end{equation}

{\bf Estimate of $J_2$:} {We will absorb a part of the term}
\[
\Big(J_p(u_h(x)-u_h(y))-J_p(u(x)-u(y))\Big)\times \Big(J_{q+1}(\delta_ h u(x)))+J_{q+1}(\delta_h u(y))\Big)\,\left(\eta(x)^2-\eta(y)^2\right)
\]
into the positive term
\[
 \Big(J_p(u_h(x)-u_h(y))-J_p(u(x)-u(y))\Big)\times \Big((J_{q+1}(\delta_hu(x))-J_{q+1}(\delta_hu(y)))\Big)\,\left(\eta(x)^2+\eta(y)^2\right).
\]

We write, noticing that $\delta_h J_p(u(x)-u(y))$ and $\delta_h u(x)-\delta_h u(y)$ have the same sign
\begin{equation}\label{neqn2}
\begin{split}
\delta_h J_p(u(x)-u(y))&=\delta_h J_p(u(x)-u(y))\left(\frac{\big(J_{q+1}(\delta_hu(x))-J_{q+1}(\delta_hu(y))\big)\delta_h J_p(u(x)-u(y))}{\big(J_{q+1}(\delta_hu(x))-J_{q+1}(\delta_h u(y))\big)\delta_h J_p(u(x)-u(y))}\right)^\frac{p-1}{p}\\
&=\left[\delta_h J_p(u(x)-u(y))\big(J_{q+1}(\delta_hu(x))-J_{q+1}(\delta_hu(y))\big)\right]^\frac{p-1}{p}\delta_h J_p(u(x)-u(y))\\
&\times\left(\frac{(\delta_h u(x)-\delta_h u(y))}{\big(J_{q+1}(\delta_h u(x))-J_{q+1}(\delta_hu(y))\big)(\delta_h u(x)-\delta_h u(y))\delta_h J_p(u(x)-u(y))}\right)^\frac{p-1}{p} \\
&=\left[\delta_h J_p(u(x)-u(y))\big(J_{q+1}(\delta_hu(x))-J_{q+1}(\delta_hu(y))\big)\right]^\frac{p-1}{p}\\
&\times \left(\frac{1}{\big(J_{{q}+1}(\delta_hu(x))-J_{q+1}(\delta_hu(y))\big)(\delta_h u(x)-\delta_h u(y))}\right)^\frac{p-1}{p} \\
&\times \delta_h J_p(u(x)-u(y))\left(\frac{\delta_h u(x)-\delta_h u(y)}{\delta_h J_p(u(x)-u(y))}\right)^\frac{p-1}{p}.
\end{split}
\end{equation}
We observe that
\begin{equation}\label{neqn4}
\begin{split}
\Big|\delta_h J_p(u(x)-u(y))\left(\frac{\delta_h u(x)-\delta_h u(y)}{\delta_h J_p(u(x)-u(y))}\right)^\frac{p-1}{p}\Big|&\leq |\delta_h u(x)-\delta_h u(y)|^\frac{p-1}{p}||\delta_h J_p(u(x)-u(y))|^\frac{1}{p}\\
&\leq {C}|\delta_h u(x)-\delta_h u(y)|^{(p-1)/p+(p-1)/p}\\
&={C}|\delta_h u(x)-\delta_h u(y)|^\frac{2(p-1)}{p}, \quad {C=C(p)}
\end{split}
\end{equation}
which follows from the $(p-1)$-H\"older regularity for $J_p$. Using Lemma \ref{PLine} and \eqref{neqn4} in \eqref{neqn2} yields
\begin{equation}\label{neqn5}
\begin{split}
|\delta_h J_p(u(x)-u(y))|&\leq {C}	\left[\delta_h J_p(u(x)-u(y))\big(J_{q+1}(\delta_hu(x))-J_{q+1}(\delta_hu(y))\big)\right]^\frac{p-1}{p}\\
&\times \left(\frac{1}{(|\delta_h u(x)|^{q-1}+|\delta_h u(y)|^{q-1})(\delta_h u(x)-\delta_h u(y))^2}\right)^\frac{p-1}{p}\\
&\times |\delta_h u(x)-\delta_h u(y)|^\frac{2(p-1)}{p}\\
&={C}\left[\delta_h J_p(u(x)-u(y))\big(J_{q+1}(\delta_hu(x))-J_{q+1}(\delta_hu(y))\big)\right]^\frac{p-1}{p}\\
&\times \left(\frac{1}{|\delta_h u(x)|^{q-1}+|\delta_h u(y)|^{q-1}}\right)^\frac{p-1}{p},\quad{ C=C(p,q).}
\end{split}
\end{equation}
Therefore, by the above estimate \eqref{neqn5} and using Young's inequality with $p$ and $p/(p-1)$, we have

\begin{equation}\label{neqn6}
\begin{split}
&|\delta_h J_p(u(x)-u(y))|\times \Big|J_{q+1}(\delta_ h u(x))+J_{q+1}(\delta_h u(y))\Big|\,\left|\eta(x)^2-\eta(y)^2\right|\\
&=|\delta_h J_p(u(x)-u(y))|\times \Big|J_{q+1}(\delta_ h u(x))+J_{q+1}(\delta_h u(y))\Big|\,\left|\eta(x)-\eta(y)\right|\left(\eta(x)+\eta(y)\right)\\
&\leq {C}\left[\delta_h J_p(u(x)-u(y))\big(J_{q+1}(\delta_hu(x))-J_{q+1}(\delta_hu(y))\big)\right]^\frac{p-1}{p}\left(\eta(x)+\eta(y)\right)\\
&\times \Big|J_{q+1}(\delta_ h u(x))+J_{q+1}(\delta_h u(y))\Big|\left(\frac{1}{(|\delta_h u(x)|^{q-1}+|\delta_h u(y)|^{q-1})}\right)^\frac{p-1}{p}|\eta(x)-\eta(y)|\\
&\leq \frac12 \delta_h J_p(u(x)-u(y))\big(J_{q+1}(\delta_hu(x))-J_{q+1}(\delta_hu(y))\big)\left(\eta(x)^\frac{p}{p-1}+\eta(y)^\frac{p}{p-1}\right)\\
&+{C} \Big|J_{q+1}(\delta_ h u(x))+J_{q+1}(\delta_h u(y))\Big|^p\left(\frac{1}{|\delta_h u(x)|^{q-1}+|\delta_h u(y)|^{q-1}}\right)^{p-1}|\eta(x)-\eta(y)|^p\\
&\leq \frac{1}{2} \delta_h J_p(u(x)-u(y))\left((J_{q+1}(\delta_hu(x))-J_{q+1}(\delta_hu(y)))\right)\left(\eta^2(x)+\eta^2(y)\right)\\
&+{C} \Big(|\delta_h u(x)|^{q+p-1}+|\delta_h u(y)|^{q+p-1}\Big)|\eta(x)-\eta(y)|^p,\quad { C=C(p,q)}
\end{split}
\end{equation}
where we used that $\eta^\frac{p}{p-1}\leq \eta^2$ since $p/(p-1)\geq 2$. Thus from \eqref{neqn6}, we conclude that
\begin{equation}\label{neqn7}
\begin{split}
J_2\leq \frac12 J_1+
{C}\left(|\delta_h u(x)|^{q+p-1}+|\delta_h u(y)|^{q+p-1}\right)|\eta(x)-\eta(y)|^p,\quad  C=C(p,q).
\end{split}
\end{equation}
Therefore, using \eqref{neqn8} and \eqref{neqn7} in \eqref{neqn0}, we obtain
\begin{equation}\label{neqn9}
\begin{split}
I&\geq \frac12 J_1-{C(p)}\left(|\delta_h u(x)|^{q+p-1}+|\delta_h u(y)|^{q+p-1}\right)|\eta(x)-\eta(y)|^p\\
&\geq C\left||\delta_h u(x)|^\frac{q-1}{2}(\delta_h u(x))-|\delta_h u (y)|^\frac{q-1}{2}(\delta_h u(y))\right|^2(\eta(x)^2+\eta(y)^2) |x-y|^{\gamma(p-2)}\\
&\quad-C(p,q) \left(|\delta_h u(x)|^{q+p-1}+|\delta_h u(y)|^{q+p-1}\right)|\eta(x)-\eta(y)|^p,\quad C=C(p,q).
\end{split}
\end{equation}
Thus \eqref{neqn9} gives
\begin{equation}\label{neqn10}
\begin{split}
&{C}\iint_{B_R\times B_R} \frac{1}{|h|^{1+\theta q}}\left||\delta_h u(x)|^\frac{q-1}{2}(\delta_h u(x))-|\delta_h u (y)|^\frac{q-1}{2}(\delta_h u(y))\right|^2(\eta(x)^2+\eta(y)^2) |x-y|^{\gamma(p-2)}d\mu\\
&\leq  \mathcal{I}_1+C\int_{B_R} \frac{|\delta_h u(x)|^{p+q-1}}{|h|^{1+\theta q}} dx,\quad C=C(N,h_0,p,q,s).
\end{split}
\end{equation}
Here we have also used that the factor $|\eta(x)-\eta(y)|^p$ cancels out the singularity of the kernel in the last term above. We now observe that with
\[
A=\frac{|\delta_h u(x)|^\frac{q-1}{2}\,\delta_h u(x)}{|h|^\frac{1+\theta\,q}{2}}\qquad \mbox{ and }\qquad B=\frac{|\delta_h u(y)|^\frac{q-1}{2}\,\delta_h u(y)}{|h|^\frac{1+\theta\,q}{2}}
\]
the convexity of $\tau\mapsto \tau^2$ implies
\[
\begin{split}
\left|A\,\eta(x)-B\,\eta(y)\right|^2&=\left|(A-B)\,\frac{\eta(x)+\eta(y)}{2}+(A+B)\,\frac{\eta(x)-\eta(y)}{2}\right|^2\\
&\le \frac{1}{2}\,|A-B|^2\,\left|\eta(x)+\eta(y)\right|^2\\
&+\frac{1}{2}\, |A+B|^2\,\left|\eta(x)-\eta(y)\right|^2\\
&\le |A-B|^2\, (\eta(x)^2+\eta(y)^2)\\
&+ (|A|^2+|B|^2)\, |\eta(x)-\eta(y)|^2.
\end{split}
\]
Taking the above inequality into account and using \eqref{neqn10}, we get the following lower bound for $\mathcal{I}_1$, with $2\sigma = sp-\gamma(p-2)$:
\[
\begin{split}
\mathcal{I}_1\ge c& \left[\frac{|\delta_h u|^\frac{q-1}{2}\,\delta_h u}{|h|^\frac{1+\theta\,q}{2}}\,\eta\right]^2_{W^{\sigma,2}(B_R)} -C\int_{B_R} \frac{|\delta_h u(x)|^{p+q-1}}{|h|^{1+\theta\,q}} dx\\
&-C\,\iint_{B_R\times B_R}\, \left(\frac{|\delta_h u(x)|^{(q+1) }}{|h|^{1+\theta\,q}}+\frac{|\delta_h u(y)|^{(q+1) }}{|h|^{1+\theta\,q}}\right)\, |\eta(x)-\eta(y)|^2\,d\mu\\
&\geq {c}\left[\frac{|\delta_h u|^\frac{q-1}{2}\,\delta_h u}{|h|^\frac{1+\theta\,q}{2}}\,\eta\right]^2_{W^{\sigma,2}(B_R)}-C\int_{B_R} \frac{|\delta_h u(x)|^{p+q-1}}{|h|^{1+\theta\,q}} dx\\
& -C\,\int_{B_R}\, \frac{|\delta_h u(x)|^{q+1}}{|h|^{1+\theta\,q}}dx,
\end{split}
\]
where we again used that $\eta$ is Lipschitz. Here $c=c(p,q)>0$ and $C=C(N,h_0,p,q,s)>0$.   Note that $\sigma\in(0,1)$ since $\gamma\in[0,1)$ and $s\in(0,1)$. By recalling that $\mathcal{I}_1+\mathcal{I}_2+\mathcal{I}_3=0$ from \eqref{morsecode} and using the estimate for $\mathcal{I}_1$, we arrive at
\begin{equation}\label{Ieq}
	\left[\frac{|\delta_h u|^\frac{q-1}{2}\,\delta_h u}{|h|^\frac{1+\theta\,q}{2}}\,\eta\right]^2_{W^{\sigma,2}(B_R)}\leq C\,\Big(\int_{B_R} \frac{|\delta_h u(x)|^{p+q-1}}{|h|^{1+\theta\,q}} +\frac{|\delta_h u(x)|^{q+1}}{|h|^{1+\theta\,q}}+|\mathcal{I}_2|+|\mathcal{I}_3|\Big),
\end{equation}
for $C=C(N,h_0,p,q,s)>0$.

{\noindent}{\bf Step 3: Estimates of the nonlocal terms $\mathcal{I}_2$ and $\mathcal{I}_3$:}
Both nonlocal terms $\mathcal{I}_2$ and $\mathcal{I}_3$ can be treated in the same way. We only estimate $\mathcal{I}_2$ for simplicity. {Using \eqref{bounds}, since $|u|\le 1$ in $B_1$, for every $x\in B_{(R+r)/2}$ and $y\in\mathbb{R}^N\setminus B_R$}, we have 
\[
\begin{split}
\Big|\big(J_p(u_h(x)-u_h(y))-J_p(u(x)-u(y))\big)\, J_{q+1}(\delta_h u(x))\Big|&\leq C\left(1+|u_h(y)|^{p-1}+|u(y)|^{p-1}\right)\,|\delta_h u(x)|^{q},
\end{split}
\]
where $C=C(p)>0$. For $x\in B_{(R+r)/2}$ we have $B_{(R-r)/2}(x)\subset B_{R}$ and thus
$$
\int_{\mathbb{R}^N\setminus B_{R}}\frac{1}{|x-y|^{N+s\,p}}\,  dy\leq \int_{\mathbb{R}^N\setminus B_\frac{R-r}{2}(x)} \frac{1}{|x-y|^{N+s\,p}}\, dy\leq C(N,h_0,p,s),
$$
by recalling that $R-r=4\,h_0$.
By using \cite[Lemma 2.2 and Lemma 3.3]{BLS}, we get for $x \in  B_{(R+r)/2}$
\[
\begin{split}
\int_{\mathbb{R}^N\setminus B_{R}} \frac{|u(y)|^{p-1}}{|x-y|^{N+s\,p}}\, dy&\le \left(\frac{2\,R}{R-r}\right)^{N+s\,p}\,\int_{\mathbb{R}^N\setminus B_R} \frac{|u(y)|^{p-1}}{|y|^{N+s\,p}} \, dy
\\
&\le \left(\frac{2\,R}{R-r}\right)^{N+s\,p}\,\int_{\mathbb{R}^N\setminus B_1} \frac{|u(y)|^{p-1}}{|y|^{N+s\,p}} \, dy+\left(\frac{2\,R}{R-r}\right)^{N+s\,p}\,R^{-N}\,\int_{B_1} |u|^{p-1}\,dy\\
&\leq C(N,h_0,p,s).
\end{split}
\]
In the last estimate we have used the bounds assumed on $u$ {in \eqref{bounds}} and $4\,h_0< R \leq 1$. The term involving $u_h$ can be estimated similarly.
Recall also that $\eta=0$ outside $B_{(R+r)/2}$. Hence, we have
\begin{equation}
\label{eq:I2est}
|\mathcal{I}_2|\leq C\int_{B_R}\frac{|\delta_h u(x)|^{q}}{|h|^{1+\theta\,q}} dx.
\end{equation}
{Similarly, we get
\begin{equation}
\label{eq:I3est}
|\mathcal{I}_3|\leq C\int_{B_R}\frac{|\delta_h u(x)|^{q}}{|h|^{1+\theta\,q}} dx.
\end{equation}
}
The combination of \eqref{Ieq}, \eqref{eq:I2est} and \eqref{eq:I3est} now implies
\begin{equation}
\label{eq:itest}
\begin{split}
	\left[\frac{|\delta_h u|^\frac{q-1}{2}\,\delta_h u}{|h|^\frac{1+\theta\,q}{2}}\,\eta\right]^2_{W^{\sigma,2}(B_R)}\leq C\,\Big(\int_{B_R}  \frac{|\delta_h u(x)|^{p+q-1}}{|h|^{1+\theta\,q}} +\frac{|\delta_h u(x)|^{q+1 }}{|h|^{1+\theta\,q}}+\frac{|\delta_h u(x)|^{q}}{|h|^{1+\theta\,q}}dx\Big),
\end{split}
\end{equation}
{where $C=C(N,h_0,p,s,q)>0.$}

{\bf Step 4: Transformation to double differences.}
For $\xi,h\in\mathbb{R}^N\setminus\{0\}$ such that $|h|,|\xi|<h_0$, we let
\[
A=u(x+h+\xi)-u(x+\xi),\qquad B=u(x+h)-u(x),\qquad \gamma=\frac{q+1}{2}.
\]
Lemma \ref{lemma:holder} implies
$$
|\delta_h \delta_\xi u|^\frac{q+1}{2}\leq C \delta_\xi |\delta_h u|^\frac{q-1}{2}\delta_h u,\quad C=C(q).
$$
Therefore,
\begin{equation}
\label{est3}
\left\|\frac{|\delta_\xi \delta_h u|^\frac{q+1}{2}}{|\xi|^\sigma\,|h|^\frac{1+\theta\,q}{2}}\right\|^{2}_{L^{2}(B_r)}
\leq C \left\|\frac{\delta_\xi \left(|\delta_h u|^\frac{q-1}{2}\,\delta_h u\right)}{|\xi|^\sigma\,|h|^\frac{1+\theta\,q}{2}}\right\|^2_{L^2(B_r)}
\leq C\,\left\|\eta\,\frac{\delta_\xi}{|\xi|^\sigma}\left(\frac{|\delta_h u|^\frac{q-1}{2}\,(\delta_h u)}{|h|^\frac{1+\theta\,q}{2}}\right)\right\|^2_{L^2(\mathbb{R}^N)},
\end{equation}
where $C=C(q)>0$. Here we used that $\eta \equiv 1$ on $B_r$. By a discrete version of Leibniz rule (see \eqref{Lrule1}), 
\begin{equation}\label{Lrule}
\eta\,\delta_\xi\left(|\delta_h u|^\frac{q-1}{2}\,(\delta_h u)\right) = \delta_\xi\left(\eta\,|\delta_h u|^\frac{q-1}{2}\,(\delta_h u)\right)-\left(|\delta_h u|^\frac{q-1}{2}\,(\delta_h u)\right)_\xi\,\delta_{\xi}\eta.
\end{equation}
Inserting {\eqref{Lrule}} into \eqref{est3} yields
\begin{equation}\label{extraest}
\left\|\frac{|\delta_\xi \delta_h u|^\frac{q+1}{2}}{|\xi|^\sigma\,|h|^\frac{1+\theta\,q}{2}}\right\|^2_{L^{2}(B_r)}
\leq C\, \left\|\frac{\delta_\xi}{|\xi|^\sigma}\left(\frac{|\delta_h u|^\frac{q-1}{2}\,(\delta_h u)\,\eta}{|h|^\frac{1+\theta\,q}{2}}  \right)\right\|^2_{L^2(\mathbb{R}^N)} 
+ C\, \left\|\frac{\delta_\xi\eta}{|\xi|^\sigma}\frac{\left(|\delta_h u|^\frac{q-1}{2}\,(\delta_h u)\right)_\xi}{|h|^\frac{1+\theta\,q}{2}}\right\|^2_{L^2(\mathbb{R}^N)},
\end{equation}
where $C=C(q)>0$. For the first term in \eqref{extraest}, we apply \cite[Proposition 2.6]{Brolin} with the choice
\[
\psi=\frac{|\delta_h u|^\frac{q-1}{2}\,(\delta_h u)\,\eta}{|h|^\frac{1+\theta\,q}{2}},
\]
and get
\begin{equation}
\label{est2}
\sup_{|\xi|>0}\left\|\frac{\delta_\xi}{|\xi|^\sigma}\frac{|\delta_h u|^\frac{q-1}{2}\,(\delta_h u)\,\eta}{|h|^\frac{1+\theta\,q}{2}}\right\|^2_{L^2(\mathbb{R}^N)}\leq C\,(1-\sigma)\left[\frac{|\delta_h u|^\frac{q-1}{2}\,(\delta_h u)\,\eta}{|h|^\frac{1+\theta\,q}{2}}\right]^2_{W^{\sigma,2}(B_R)}, 
\end{equation}
where $C=C(N,h_0,\sigma)>0$. Here we also used that $\frac{R+r}{2} + 2h_0 = R$.
\par
As for the second term in \eqref{extraest}, we observe that for every $0<|\xi|<h_0$
\begin{equation}
\label{arrivotre}
\begin{split}
\left\|\frac{\delta_\xi\eta}{|\xi|^\sigma}\frac{\left(|\delta_h u|^\frac{q-1}{2}\,(\delta_h u)\right)_\xi}{|h|^\frac{1+\theta\,q}{2}}\right\|^2_{L^2(\mathbb{R}^N)}&\leq
 C\,\left\|\frac{\left(|\delta_h u|^\frac{q-1}{2}\,(\delta_h u)\right)_\xi}{|h|^\frac{1+\theta\,q}{2}}\right\|^2_{L^2(B_{\frac{R+r}{2}+h_0})}\\
&\leq C\,\int_{B_{\frac{R+r}{2}+2h_0}}\frac{|\delta_h u|^{q+1}}{|h|^{1+\theta\,q}} dx\\
&\leq C\,\int_{B_{R}}\frac{|\delta_h u|^{q+1}}{|h|^{1+\theta\,q}} dx,
\end{split}
\end{equation}
where $C=C(N,h_0,s)>0$. Here we have used the estimate of $\nabla \eta$. 
\par
From \eqref{extraest}, \eqref{est2}, and \eqref{arrivotre} we get for any $0 < |\xi| < h_0$
\begin{equation}
\label{est4}
\left\|\frac{\delta_\xi \delta_h u}{|\xi|^\frac{2\sigma}{q+1}|h|^\frac{1+\theta\,q}{(q+1)}}\right\|^{q+1}_{L^{q+1}(B_r)}\leq C\,\left[\frac{|\delta_h u|^\frac{q-1}{2}\,(\delta_h u)}{|h|^\frac{1+\theta\,q}{2}}\eta\right]^2_{W^{\sigma,2}(B_R)}+C\,\int_{B_{R}}\frac{|\delta_h u|^{q+1}}{|h|^{1+\theta\,q}} dx,
\end{equation}
with $C=C(N,h_0,s,q,\sigma)>0$.
We then choose $\xi=h$ and take the supremum over $h$ for $0<|h|< h_0$. Then \eqref{est4} together with \eqref{eq:itest} imply
\begin{equation}\label{almostfinal}
\sup_{0<|h|< h_0}\int_{B_r}\left|\frac{\delta^2_h u}{|h|^{\frac{2\sigma}{q+1}+\frac{1+\theta\,q}{(q+1)}}}\right|^{q+1}\,dx\leq C\,\Big(\int_{B_R}  \frac{|\delta_h u(x)|^{p+q-1}}{|h|^{1+\theta\,q}} +\frac{|\delta_h u(x)|^{q+1}}{|h|^{1+\theta\,q}}+\frac{|\delta_h u(x)|^{q}}{|h|^{1+\theta\,q}}dx\Big),
\end{equation}
where $C=C(N,h_0,p,s,q,\sigma)>0$.

Now we { choose $\theta=\alpha-1/q$ and observe that since $\|u\|_{L^\infty(B_1)}\leq 1$ from \eqref{bounds}, the assumption that $4h_0<R\leq 1-5h_0$ and the fact that $q\leq q+p-1\leq q+1$, implies that the first and the second terms in the right hand side of \eqref{almostfinal} can be estimated by the third one. Recalling also that $2\sigma = sp-\gamma(p-2)$, this yields}
\begin{equation}\label{final}
\sup_{0<|h|< h_0}\int_{B_r}\left|\frac{\delta^2_h u}{|h|^{\frac{sp-\gamma(p-2)+\alpha q}{q+1}}}\right|^{q+1}\,dx\leq C\,\int_{B_{R+h_0}}   \frac{|\delta_h u(x)|^{q}}{|h|^{\alpha\,q}}dx,
\end{equation}
where $C=C(N,h_0,p,s,q,\gamma)>0$. Since $\alpha<1$, {taking into account \eqref{asm1}} and using the second estimate of \cite[Lemma 2.6]{BLS} we replace the first order difference quotient in the {right-hand side of \eqref{final}} with  a second order difference quotient. Then \eqref{final} transforms into the desired inequality \eqref{pro1}, upon recalling the relations between $R,r$ and $h_0$.
\end{proof}

\subsection{Improved H\"older regularity}
We can now iterate the improved Besov-type regularity to obtain an improved H\"older regularity.
\begin{prop}
\label{prop:improve2} 
Assume $1<p<2$, $0<s<1$ and {$\gamma\in [0,1)$}. Let  $u\in W^{s,p}_{\rm loc}(B_2)\cap L^{p-1}_{s\,p}(\mathbb{R}^N)$ be a local weak solution of $(-\Delta_p)^s u=0$ in $B_2$. Suppose that 
\begin{equation*}
\label{boundsimp}
\|u\|_{L^\infty(B_1)}\leq 1, \qquad \mathrm{Tail}_{p-1,s\,p}(u;0,1)^{p-1}=\int_{\mathbb{R}^N\setminus B_1} \frac{|u(y)|^{p-1}}{|y|^{N+s\,p}}\,dy\leq 1\quad  \mbox{ and }\quad [u]_{C^\gamma(B_1)}\leq 1.
\end{equation*} 
Let $\tau =\min (sp-\gamma(p-2),1)$. Then for any $\e\in (0,\tau)$, we have
$$
 [u]_{C^{\tau-\e}(B_\frac12)}\leq C(s,p,\e,N,\gamma).
$$
\end{prop}
\begin{proof}
Take $0<\e<\tau$ and choose $q$ so that 
$$
 \tau-\frac{\e}{2}-\frac{N}{q}>\tau-\e>0.
$$
Then we define the sequence of exponents  
\[
\alpha_0=0, \quad \alpha_i=\frac{sp-\gamma(p-2)+\alpha_{i-1} q}{q+1},\qquad i=0,\dots,i_\infty,
\]
where we choose {$i_\infty\geq 1$} such that 
$$
\alpha_{i_\infty-1}<\tau-\e/2\leq \alpha_{i_\infty}.
$$
Note that this is possible since the sequence of exponents $\alpha_i$ are increasing towards $sp-\gamma(p-2)$. Define also 
$$
h_0=\frac{1}{64\,i_\infty},\qquad R_i=\frac{7}{8}-4\,(2i+1)\,h_0=\frac{7}{8}-\frac{2i+1}{16\,i_\infty},\qquad \mbox{ for } i=0,\dots,i_\infty.
$$
We note that 
\[
R_0+4\,h_0=\frac{7}{8}\qquad \mbox{ and }\qquad R_{i_\infty-1}-4\,h_0=\frac{3}{4}.
\] 
By applying Proposition \ref{prop:improve} and with $R=R_i$ and observing that $R_i-4\,h_0=R_{i+1}+4\,h_0$,
we obtain the iterative scheme of inequalities
\[
\left\{\begin{array}{rcll}
\sup\limits_{0<|h|< h_0}\left\|\dfrac{\delta^2_h u}{|h|^{\alpha_1}}\right\|_{L^{q}(B_{R_1+4h_0})}&\leq& C\,\sup\limits_{0<|h|< h_0}\left(\left\|\delta^2_h u \right\|_{L^q(B_{7/8})}+1\right)\\
&&&\\
\sup\limits_{|h|\leq h_0}\left\|\dfrac{\delta^2_h u}{|h|^{\alpha_{i+1}}}\right\|_{L^{q}(B_{R_{i+1}+4h_0})}&\leq& C\,\sup\limits_{0<|h|< h_0}\left(\left\|\dfrac{\delta^2_h u }{|h|^{\alpha_i}}\right\|_{L^{q}(B_{R_i+4h_0})}+1\right),&\mbox{ for } i=1,\ldots,i_\infty-2,
\end{array}
\right.
\]
and finally
$$
\sup_{0<|h|< {h_0}}\left\|\frac{\delta^2_h u}{|h|^{\alpha_{i_\infty}}}\right\|_{L^{q}(B_{3/4})}=\sup_{0<|h|< h_0}\left\|\frac{\delta^2_h u}{|h|^{\alpha_{i_\infty}}}\right\|_{L^{q}(B_{R_{i_\infty-1}-4h_0})}\leq C\sup_{0<|h|< h_0}\left(\left\|\frac{\delta^2_h u }{|h|^{\alpha_{i_\infty-1}}}\right\|_{L^{q}(B_{R_{i_\infty-1}+4h_0})}+1\right).
$$
Here $C=C(N,\e,p,s,\gamma)>0$. Also,
\begin{align*}\label{eq:1sttofrac}
\sup\limits_{0<|h|< h_0}\left\|\delta^2_h u \right\|_{L^q(B_{7/8})}&\leq {3}\|u\|_{L^{\infty}(B_{1})}\leq {3}.
\end{align*}
Hence, the iterative scheme of inequalities leads us to
\begin{equation}\label{neqn11}
\sup_{0<|h|< {h_0}}\left\|\frac{\delta^2_h u}{|h|^{\alpha_{i_\infty}}}\right\|_{L^{q}(B_{3/4})}\leq C(N,\e,p,s,\gamma).
\end{equation}
Using $\alpha_{i_\infty}\geq \tau-\e/2$ in \eqref{neqn11} implies
\begin{equation}\label{neqn12}
\sup_{0<|h|< {h_0}}\left\|\frac{\delta^2_h u}{|h|^{\tau-\e/2}}\right\|_{L^{q}(B_{3/4})}\leq C(N,\e,p,s,\gamma).
\end{equation}
Take now $\chi\in C_0^\infty(B_{5/8})$ such that
$$
0\leq \chi\leq 1, \qquad |\nabla \chi|\leq C,\qquad |D^2 \chi|\leq C \text{ in }B_\frac{5}{8}\qquad\text{ and }\qquad \chi =1 \text{ in }B_{\frac{1}{2}}.
$$
In particular we have for all $|h| > 0$
$$
\frac{|\delta_h\chi|}{|h|}\leq C,\qquad \frac{|\delta^2_h\chi|}{{|h|^2}}\leq C.
$$
We also recall that
$$
\delta^2_h (u\,\chi)=\chi_{2h}\,\delta^2_h u+2\,\delta_h u\, \delta_h \chi_h+u\,\delta^2_h\chi.
$$
Hence, for $0<|h|< h_0$, using the above properties of $\chi$ and \eqref{neqn12}, we have
\begin{align}\label{Neq}
[u\,\chi]_{\mathcal{B}^{\tau-\e/2,q}_\infty(\mathbb{R}^N)}=\left\|\frac{\delta^2_h (u\,\chi)}{|h|^{\tau-\e/2}}\right\|_{L^{q}(\mathbb{R}^N)}&\leq C\,\left(\left\|\frac{\chi_{2h}\,\delta^2_h u}{|h|^{\tau-\e/2}}\right\|_{L^{q}(\mathbb{R}^N)}+\left\|\frac{\delta_h u\,\delta_h\chi}{|h|^{\tau-\e/2}}\right\|_{L^{q}(\mathbb{R}^N)}+\left\|\frac{u\,\delta^2_h\chi}{|h|^{\tau-\e/2}}\right\|_{L^{q}(\mathbb{R}^N)}\right) \nonumber\\
&\leq C\,\left(\left\|\frac{\delta^2_h u}{|h|^{\tau-\e/2}}\right\|_{L^{q}(B_{5/8+2\,h_0})}+\|\delta_h u\|_{L^{q}(B_{5/8+h_0})}+\|u\|_{L^{q}(B_{5/8+2h_0})}\right) \\
&\leq C\,\left(\left\|\frac{\delta^2_h u}{|h|^{\tau-\e/2}}\right\|_{L^{q}(B_{3/4})}+\|u\|_{L^{q}(B_{3/4})}\right)\leq C(N,\e,p,s,\gamma). \nonumber
\end{align}
Thus by \eqref{Neq} and Lemma \ref{emb1}, we have
\begin{equation*}\label{Neq2}
[u\,\chi]_{\mathcal{N}_\infty^{\tau-\frac{\e}{2},q}(\mathbb{R}^N)}\leq  C(N,\e,q)\,[u\,\chi]_{\mathcal{B}_\infty^{\tau-\e/2,q}(\mathbb{R}^N)}\leq C(N,\e,p,s,\gamma).
\end{equation*}
Finally, thanks to the choice of $q$ we have
\[
 \tau-\e<\tau-\frac{\e}{2}-\frac{N}{q}.
\] 
We may therefore apply Theorem \ref{emb2} with  $\beta=\tau-\frac{\e}{2}$ and $\alpha=\tau-\e$ to obtain
$$
[u]_{C^{\tau-\e}(B_{1/2})}= [u\,\chi]_{C^{\tau-\e}(B_{1/2})}\leq C\left([u\,\chi]_{\mathcal{N}_\infty^{\tau-\frac{\e}{2},q}(\mathbb{R}^N)}\right)^{\frac{(\tau-\e)\,q+N}{(\tau-\frac{\e}{2})\,q}}\,\left(\|u\,\chi\|_{L^q(\mathbb{R}^N)}\right)^\frac{\frac{q\e}{2}-N}{(\tau-\frac{\e}{2})\,q}\leq C(N,\e,p,s,\gamma).
$$
This concludes the proof.
\end{proof}

\subsection{Final H\"older regularity}
We first prove a normalized version of Theorem \ref{teo:1}. This is accomplished by iterating the previously obtained improved H\"older regularity.
\begin{Theorem}[Almost $sp/(p-1)$-regularity]
\label{teo:localalmost}
Let $1<p<2$ and $0<s<1$. Suppose $u\in W^{s,p}_{\rm loc}(B_2)\cap L^{p-1}_{s\,p}(\mathbb{R}^N)$ is a local weak solution of 
\[
(-\Delta_p)^s u=0\qquad \mbox{ in }B_2
\] 
such that
$$
\|u\|_{L^\infty(B_1)}\leq 1\qquad \mbox{ and }\qquad \mathrm{Tail}_{p-1,s\,p}(u;0,1)^{p-1}=\int_{\mathbb{R}^N\setminus B_1} \frac{|u(y)|^{p-1}}{|y|^{N+s\,p}}\,  dy\leq 1.
$$
Then for any $\e\in(0,\Gamma)$, there is $\sigma(\e,N,s,p)>0$ such that $u\in C^{\Gamma-\e}(B_\sigma)$, where
$$
\Gamma = \min(sp/(p-1),1).
$$
Moreover, 
$$
[u]_{C^{\Gamma-\e}(B_\sigma)}\leq C(s,p,\e,N).
$$
\end{Theorem}
\par
\begin{proof} The idea is to apply Proposition \ref{prop:improve2} iteratively. Take $\e\in(0,\Gamma)$ and define
$$
\gamma_0=0,\qquad \gamma_{i+1}=sp-\gamma_i(p-2)-\frac{\e(p-1)}{2}.
$$
Then $\gamma_i$ is an increasing sequence and $\gamma_i\to sp/(p-1)-\e/2$, as $i\to\infty$.  Define also
$
v_i(x)=u(2^{-i} x)
$
and
$$
M_i = \|u\|_{L^\infty(B_{2^{-i}})}+ \mathrm{Tail}_{p-1,s\,p}(u;0,2^{-i})+2^{-i\gamma_i}[u]_{C^{\gamma_i}(B_{2^{-i}})}\leq C(i)(1+[u]_{C^{\gamma_i}(B_{2^{-i}})}).
$$
{It is clear that there is {$i_\infty=i_\infty(\e)\in\mathbb{N}$} such that $\gamma_{i_\infty}\geq \Gamma-\e$ and $\gamma_{i_\infty-1}<1$.} Now we apply Proposition \ref{prop:improve2} to $v_i/M_i$ successively with $\gamma=\gamma_i$ and $\e$ replaced by $\frac{\e(p-1)}{2}$ and obtain 
\[
\begin{split}
2^{\gamma_1}[v_1]_{C^{\gamma_1}(B_1)}&=[v_0]_{C^{\gamma_1}(B_\frac12)}\leq C(s,p,\e,N)\\
2^{\gamma_{i+1}}[v_{i+1}]_{C^{\gamma_{i+1}}(B_1)}&=[v_i]_{C^{\gamma_{i+1}}(B_\frac12)}\leq M_i C(s,p,\e,N)\leq C(s,p,\e,N)(1+[v_i]_{C^{\gamma_i}(B_1)})\\
&\cdots\\
 2^{\Gamma-\e}[v_{i_{\infty}}]_{C^{\Gamma-\e}(B_\frac12)}&\leq [v_{i_{\infty}-1}]_{C^{\min(\gamma_{i_{\infty}},1-\frac{\e(p-1)}{2})}(B_\frac12)}\leq C(s,p,\e,N)(1+[v_{i_{\infty-1}}]_{C^{\gamma_{i_{\infty-1}}}(B_{1})}).
\end{split}
\]
Note that at every iteration step we get the estimate multiplied by a constant $C(s,p,\e,N)$. Hence, by scaling back we obtain 
$$
[u]_{C^{\Gamma-\e}(B_{2^{-i_\infty-1}})}=2^{i_\infty(\Gamma-\e)}[v_{i_{\infty}}]_{C^{\Gamma-\e}(B_\frac12)}\leq C(s,p,\e,N).
$$
This is the desired result with $\sigma = 2^{-i_\infty-1}$.
\end{proof}
The proof of the main H\"older regularity now easily follows. We spell out the details.
\begin{proof}[~Proof of Theorem \ref{teo:1}]
By Theorem 1.1 in \cite{DKP}, $u\in L^{\infty}_{\mathrm{loc}}(B_{2R}(x_0))$, so the assumption on the boundedness makes sense. Assume for simplicity that $x_0=0$ and let
\[
u_R(x):=\frac{1}{\mathcal{M}_R}\,u(R\,x),\qquad \mbox{ for }x\in B_2,
\]
where
\[
\mathcal{M}_R=\|u\|_{L^\infty(B_{R})}+\mathrm{Tail}_{p-1,s\,p}(u;0,R)>0.
\]
Then $u_R$ is a local weak solution of $(-\Delta_p)^s u=0$ in $B_2$ and satisfies
\begin{equation*}
\label{assumption}
\|u_R\|_{L^\infty(B_1)}\leq 1,\qquad \int_{\mathbb{R}^N\setminus B_1}\frac{|u_R(y)|^{p-1}}{|y|^{N+s\,p}}\,  dy\leq 1.
\end{equation*}
By Theorem \ref{teo:localalmost}, $u_R$ satisfies the estimate
\[
[u_R]_{C^{\Gamma-\e}(B_{\sigma })}\leq C.
\]
By scaling back, we obtain the desired estimate.

\end{proof}

\begin{Remark}\label{rem:covering}
We note that as usual, once a local estimate of the spirit of Theorem \ref{teo:1} is obtained, one may obtain a similar estimate for any ball strictly contained in $\Omega$ by a standard covering argument. See for instance Remark 4.3 in \cite{BLS} for a proof of such a fact.
\end{Remark}

\section{The inhomogeneous equation}
 \label{sec:inhomo}
In this section, we treat the regularity for the inhomogeneous equation by approximation.
\subsection{Basic regularity for the inhomogeneous equation}
For our purpose, we need a uniform H\"older estimate for \emph{some} exponent $\alpha\in (0,1)$. The argument used to prove this is inspired by \cite{BLS} and \cite{KMS}.

We begin with a Caccioppoli estimate for solutions to the inhomogeneous equation.
\begin{Lemma}
\label{lm:cacc1}
Let $1<p<2$ and $0<s<1$. Suppose $\Omega\subset\mathbb{R}^N$ is an open and bounded set such that $B_r(x_0)\Subset B_R(x_0)\subset \Omega$. For $f\in L^q_{\rm loc}(\Omega)$, with
\[
q\ge (p^*_s)'\quad \mbox{ if } s\,p\not =N\qquad \mbox{ or }\qquad q>1\quad \mbox{ if }s\,p=N,
\]
we consider a local weak solution $u\in W^{s,p}_{\rm loc}(\Omega)\cap L^{p-1}_{s\,p}(\mathbb{R}^N)$ of the equation
\[
(-\Delta_p)^s u=f,\qquad \mbox{ in }\Omega.
\]
Then
\begin{equation*}
\begin{split}
[u]^p_{W^{s,p}(B_{r}(x_0))}&\leq {C{\Big(\frac{R}{R-r}\Big)^{N+sp+p}}R^{N-sp}\Big\{\|u\|^p_{L^\infty(B_R(x_0))}+(\mathrm{Tail}_{p-1,sp}(u;x_0,R))^{p-1}\|u\|_{L^\infty(B_R(x_0))}\Big\}}\\&\qquad\qquad\qquad+ CR^{\frac{N}{q'}} \|u\|_{L^\infty(B_R(x_0))}\|f\|_{L^q(B_R(x_0))}
\end{split}
\end{equation*}
for some positive constant $C=C(N,s,p)$.
\end{Lemma}
\begin{proof} We only perform the proof for $u_+$. Proceeding exactly as in the proof of Corollary 3.6 in \cite{BP}, we take a smooth function $\phi$ such that $\phi=1$ in $B_r(x_0)$, {$0\leq\phi\leq 1$ in $B_\frac{R+r}{2}(x_0)$,} $\phi=0$ outside $B_{(R+r)/2}(x_0)$ and $|\nabla \phi|\leq C/(R-r)$. Testing the equation with {$\phi^p u_+$} as in the proof of Corollary 3.6 in \cite{BP} we obtain 
\begin{equation}\label{eng1}
\begin{split}
&\int_{B_R(x_0)}\int_{B_R(x_0)}|u_+(x)\phi(x)-u_+(y)\phi(y)|^p\,d\mu\\
&\leq C\Big(\frac{R}{R-r}\Big)^{N+sp+p}R^{-sp}\left(\|u_+\|_{L^p(B_R(x_0))}^p+(\textrm{Tail}_{p-1,sp}(u_+;x_0,R))^{p-1}\|u_+\|_{L^1(B_R(x_0))}\right)
\\&+\int_{B_R(x_0)} fu_+ \, dx
\end{split}
\end{equation}
for some positive constant $C=C(N,s,p)$. We note that since $\phi=1$ in $B_r(x_0)$, the left hand side of \eqref{eng1} can be bounded from below by $[u_+]_{W^{s,p}(B_r(x_0))}^p$. In addition, the first two terms can be estimated as 
\begin{equation*}
\begin{split}
&C\Big(\frac{R}{R-r}\Big)^{N+sp+p}R^{-sp}\left(\|u_+\|_{L^p(B_R(x_0))}^p+(\mathrm{Tail}_{p-1,sp}(u_+;x_0,R))^{p-1}\|u_+\|_{L^1(B_R(x_0))}\right)\\
&\leq {C{\Big(\frac{R}{R-r}\Big)^{N+sp+p}}R^{N-sp}\Big\{\|u\|^p_{L^\infty(B_R(x_0))}+(\mathrm{Tail}_{p-1,sp}(u;x_0,R))^{p-1}\|u\|_{L^\infty(B_R(x_0))}\Big\}},
\end{split}
\end{equation*}
which matches the first two terms in the statement of the lemma. It remains to estimate the term involving $f$.  By H\"older's inequality, we have
\[
\begin{split}
\int_{B_R(x_0)} |f u_+| \,dx &\leq \|u_+\|_{L^{q'}(B_R(x_0))} \|f\|_{L^q(B_R(x_0))}\\
&\leq CR^\frac{N}{q'} \|u_+\|_{L^{\infty}(B_R(x_0))} \|f\|_{L^q(B_R(x_0))}\\
&\leq CR^{\frac{N}{q'}}\|u\|_{L^\infty(B_R(x_0))} \|f\|_{L^q(B_R(x_0))}.
\end{split}
\]
This completes the proof.
\end{proof}
The following results provides stability for the inhomogeneous equation.
\begin{Lemma}
\label{lm:1}
Let $1<p<2$, $0<s<1$ and $\Omega\subset\mathbb{R}^N$ be an open and bounded set.  Suppose that $u\in W^{s,p}_{\rm loc}(\Omega)\cap L^{p-1}_{s\,p}(\mathbb{R}^N)$ is a local weak solution of the equation
\[
(-\Delta_p)^s u=f,\qquad \mbox{ in }\Omega,
\]
where $f\in L^q_{\rm loc}(\Omega)$, with
\[
q\ge (p^*_s)'\quad \mbox{ if } s\,p\not =N\qquad \mbox{ or }\qquad q>1\quad \mbox{ if }s\,p=N.
\]
Let $B=B_{\sigma r}\Subset B'=B_r\Subset\Omega $ be a pair of concentric balls and take $v\in X^{s,p}_u(B,B')$ to be the unique weak solution of 
\[
\left\{\begin{array}{rcll}
(-\Delta_p)^s\,v&=&0,&\mbox{ in }B,\\
v&=&u,& \mbox{ in }\mathbb{R}^N\setminus B.
\end{array}
\right.
\]
For any  $\e\in (0,1/2)$ we have
\begin{equation}
\label{prima}
[u-v]^p_{W^{s,p}(\mathbb{R}^N)}\le C\e^\frac{p-2}{p-1}\,|B|^{\frac{p'}{q'}-\frac{p}{p-1}\,\frac{N-s\,p}{N\,p}}\,\left(\int_B |f|^{q}\,dx\right)^\frac{p'}{q}
+\e [u]^p_{W^{s,p}(B')},
\end{equation}
and
\begin{equation}
\label{seconda}
\fint_{B} |u-v|^p\,dx\le  C\e^\frac{p-2}{p-1}\,|B|^{\frac{p'}{q'}-\frac{p}{p-1}\,\frac{N-s\,p}{N\,p}+\frac{s\,p}{N}-1}\,\left(\int_B |f|^{q}\,dx\right)^\frac{p'}{q}
+\e |B|^{\frac{s\,p}{N}-1}[u]^p_{W^{s,p}(B')},
\end{equation}
whenever $s\,p\not =N$ and for a constant $C=C(N,p,s,\sigma)>0$.
\par
If instead $s\,p=N$, a similar estimate holds with $N\,p/(N-s\,p)$ replaced by an arbitrary exponent $m<\infty$ and the constant $C$ depending on $m$ as well.
\end{Lemma}
\begin{proof} We only perform the proof in the case $sp<N$. We first observe that the existence of $v$ is guaranteed by Theorem 2.12 in \cite{BLS}, since $u\in W^{s,p}(B')\cap L^{p-1}_{s\,p}(\mathbb{R}^N)$.
By using the weak formulations of the equations solved by $u$ and $v$ with the test function $w=u-v$, we get
\[
\iint_{\mathbb{R}^N\times\mathbb{R}^N} \frac{\Big(J_p(u(x)-u(y))-J_p(v(x)-v(y))\Big)\,\big(w(x)-w(y)\big)}{|x-y|^{N+s\,p}}\,dx\,dy={\int_{B}} f w\,dx.
\]
Let $a=u(x), b=u(y), c=v(x)$ and $d=v(y)$. By Lemma B.4 in \cite{BP}, 
\begin{equation} 
\label{eq:ellipticest}
\begin{split}
&\Big(J_p(a-b)-J_p(c-d)\Big)\Big((a-c)-(b-d)\Big)\\
&\geq(p-1)|(a-c)-(b-d)|^2 \left(|a-b|^2+|c-d|^2\right)^\frac{p-2}{2}.
\end{split}
\end{equation}
By \eqref{eq:ellipticest} and H\"older's inequality together with some trivial manipulations, we obtain
\begin{equation}
\label{eq:abcest1}
\begin{split}
&\iint_{B'\times B'} |(a-c)-(b-d)|^p d\mu\\
&\leq {C(p)}\left(\iint_{B'\times B'} \Big(J_p(a-b)-J_p(c-d)\Big)\Big((a-c)-(b-d)\Big)\,{d\mu}\right)^\frac{p}{2}\\
&\times \left(\iint_{B'\times B'}(|a-b|^{2}+|c-d|^{2})^\frac{p}{2} d\mu\right)^\frac{2-p}{2},\\
&\leq {C(p)}\left(\iint_{B'\times B'} \Big(J_p(a-b)-J_p(c-d)\Big)\Big((a-c)-(b-d)\Big){\,d\mu}\right)^\frac{p}{2}\\
&\times \left(\iint_{B'\times B'}(|a-b|^{p}+|(a-c)-(b-d)|^{p}) d\mu\right)^\frac{2-p}{2}.
\end{split}
\end{equation}
Recalling the above choices of $a,b,c,d$ and using H\"older's inequality together with the localized Sobolev inequality (cf. Proposition 2.3 in \cite{BP}), from \eqref{eq:abcest1} we have
\[
\begin{split}
[w]^p_{W^{s,p}(B')}
&\le C\,\left(\int_{B} |fw|\,dx\right)^\frac{p}{2}\left([w]^p_{W^{s,p}(B')}+[u]^p_{W^{s,p}(B')}\right)^\frac{2-p}{2}\\
&\le C\,\left\{\|f\|_{L^q(B)}\,\|w\|_{L^{q'}(B)}\right\}^{\frac{p}{2}}\left([w]^p_{W^{s,p}(B')}+[u]^p_{W^{s,p}(B')}\right)^\frac{2-p}{2}\\
&\le C\left\{|B|^{\frac{1}{q'}-\frac{1}{p^*_s}}\,\|f\|_{L^q(B)}\,\|w\|_{L^{p^*_s}(B)}\right\}^\frac{p}{2}\left([w]^p_{W^{s,p}(B')}+[u]^p_{W^{s,p}(B')}\right)^\frac{2-p}{2}\\
&\le C\,\left\{\frac{|B'|}{|B|}\frac{|B|^\frac{1}{N}}{|B'|^\frac{1}{N}-|B|^\frac{1}{N}}\left(\frac{|B'|^\frac{1}{N}}{|B'|^\frac{1}{N}-|B|^\frac{1}{N}}\right)^{{sp}}+1\right\}^{\frac{1}{2}}\\
&\times \left\{|B|^{\frac{1}{q'}-\frac{1}{p^*_s}}\,\|f\|_{L^q(B)}\,[w]_{W^{s,p}(B')}\right\}^\frac{p}{2}\left([w]^p_{W^{s,p}(B')}+[u]^p_{W^{s,p}(B')}\right)^\frac{2-p}{2}\\
& \leq C\left\{|B|^{\frac{1}{q'}-\frac{1}{p^*_s}}\,\|f\|_{L^q(B)}\,[w]_{W^{s,p}(B')}\right\}^\frac{p}{2}\left([w]^p_{W^{s,p}(B')}+[u]^p_{W^{s,p}(B')}\right)^\frac{2-p}{2},
\end{split}
\]
where $C=C(N,p,s,\sigma)$. By Young's inequality with exponent 2 this implies
$$
[w]^p_{W^{s,p}(B')}\leq C \left\{|B|^{\frac{1}{q'}-\frac{1}{p^*_s}}\,\|f\|_{L^q(B)}\right\}^p\left([w]^p_{W^{s,p}(B')}+[u]^p_{W^{s,p}(B')}\right)^{2-p},
$$
with $C=C(N,p,s,\sigma)$. Using Young's inequality with exponents $1/(p-1)$ and $1/(2-p)$ we obtain 
$$
[w]^p_{W^{s,p}(B')}\leq C\e^\frac{p-2}{p-1} \left\{|B|^{\frac{1}{q'}-\frac{1}{p^*_s}}\,\|f\|_{L^q(B)}\right\}^{p'}+\varepsilon[u]^p_{W^{s,p}(B')},
$$
where $C=C(N,p,s,\sigma)$.

Using the above estimate and arguing as in the proof of Proposition 2.3 in \cite{BP}, we have
\begin{equation*}\label{gest}
\begin{split}
[w]^{p}_{W^{s,p}(\mathbb{R}^N)}&=[w]^{p}_{W^{s,p}(B')}+2\int_{B'}\int_{\mathbb{R}^N\setminus B'} |w(x)|^p \,d\mu\\
&\leq C\left(1+\frac{|B'|}{|B|}\frac{|B|^\frac{1}{N}}{|B'|^\frac{1}{N}-|B|^\frac{1}{N}}\left(\frac{|B'|^\frac{1}{N}}{|B'|^\frac{1}{N}-|B|^\frac{1}{N}}\right)^{{sp}}\right)[w]^p_{W^{s,p}(B')}\\
&\leq C{\e^\frac{p-2}{p-1}}\left\{|B|^{\frac{1}{q'}-\frac{1}{p^*_s}}\,\|f\|_{L^q(B)}\right\}^{p'}+\varepsilon[u]^p_{W^{s,p}(B')},
\end{split}
\end{equation*}
for some constant $C=C(N,p,s,\sigma)>0$, which in turn gives \eqref{prima}. 

\par
Estimate \eqref{seconda} now follows by applying Poincar\'e's inequality in \eqref{prima}, see Lemma \ref{BLSprop}.
\end{proof}

In the lemma below we obtain a Campanato estimate. Here we use the notation 
$$
\overline u_{x_0,r}=\fint_{B_r(x_0)} u\,dx
$$
to denote the average of $u$ in $B_r(x_0)$.

\begin{Lemma}[Decay transfer]
\label{lm:transfer}
Let $1<p<2$, $0<s<1$ and $\Omega\subset\mathbb{R}^N$ be an open and bounded set. Suppose that $u\in W^{s,p}_{\rm loc}(\Omega)\cap L^{p-1}_{s\,p}(\mathbb{R}^N)$ is a local weak solution of the equation
\[
(-\Delta_p)^s u=f,\qquad \mbox{ in }\Omega,
\]
where $f\in L^{q}_{\rm loc}(\Omega)$
with
\[
q\ge (p^*_s)'\quad \mbox{ if } s\,p\not =N\qquad \mbox{ or }\qquad q>1\quad \mbox{ if }s\,p=N.
\]
If $B_{4R}(x_0)\Subset\Omega$ such that {$0<R\leq 1$}, then there is $\alpha\in (0,1)$ such that for any $\e\in (0,\frac12)$  we have
\[
\begin{split}
&\fint_{B_r(x_0)} |u-\overline u_{x_0,r}|^p\,dx\\
&\le C\Big\{\mathrm{Tail}_{p-1,sp}(u,x_0,4R)^p+\|u\|_{L^\infty(B_{4R(x_0)})}^p+\|f\|_{L^q(B_{4R(x_0)})}^{p'}+1\Big\}\Big\{\big(\frac{R}{r}\big)^N \e^\frac{p-2}{p-1}R^\gamma+\big(\frac{R}{r}\big)^N\e+\big(\frac{r}{R}\big)^{\alpha p}\Big\}
\end{split}
\] 
for every $0<r\le R$. Here
\begin{equation}
\label{gamma}
\gamma:=\left\{\begin{array}{cc}
s\,p\,p'+N\,\left(\dfrac{p'}{q'}-\dfrac{1}{p-1}-1\right),& \mbox{ if }s\,p\not=N,\\
&\\
N\,p'\,\left(\dfrac{1}{q'}-\dfrac{1}{m}\right),& \mbox{ for an arbitrary } q'< m<\infty, \mbox{ if } s\,p=N,
\end{array}
\right.
\end{equation}
and $C=C(N,s,p,q,m)>0$.
\end{Lemma}
\begin{proof} The proof is the same as the proof of Lemma 3.5 in \cite{BLS}, except for the last term that appears when applying \eqref{lm:transfer} in the present case $p< 2$. We present some details in the case $s\,p<N$.
In order to estimate this extra term, we use that the boundedness of $u$ together with Lemma \ref{lm:cacc1} applied to the balls $B_{7R/2}(x_0)$ and $B_{4R}(x_0)$ gives
\[
\begin{split}
R^{sp-N}[u]_{W^{s,p}(B_{7R/2}(x_0))}^p &\leq C\Big\{\|u\|^{p}_{L^\infty(B_{4R}(x_0))}+\mathrm{Tail}_{p-1,sp}(u;x_0,4R)^{p-1}\|u\|_{L^\infty(B_{4R}(x_0))}\Big\}\\
&\qquad\qquad\qquad +CR^{\frac{N}{q'}+sp-N}\|f\|_{L^q(B_{4R}(x_0))}\|u\|_{L^\infty(B_{4R}(x_0))}\\
&\leq C\Big\{\|u\|_{L^\infty(B_{4R}(x_0))}^p+\mathrm{Tail}_{p-1,sp}(u;x_0,4R)^p\Big\}+C\|f\|_{L^q(B_{4R}(x_0))}^{p'},
\end{split}
\]
for some constant $C=C(N,s,p)$, where we also used Young's inequality and that $\frac{N}{q'}+sp-N>0$, $\gamma>0$ and $0<r\leq R\leq 1$.
\end{proof}
We are now ready to prove H\"older regularity.
\begin{Theorem}
\label{teo:holderf}
Let $1<p<2$ and $0<s<1$. Suppose that $u\in W^{s,p}_{\rm loc}(\Omega)\cap L^{p-1}_{s\,p}(\mathbb{R}^N)$ is a local weak solution of the equation
\[
(-\Delta_p)^s u=f,\qquad \mbox{ in }\Omega,
\]
for $f\in L^q_{\rm loc}(\Omega)$ with 
\[
\left\{\begin{array}{lr}
q>{\frac{N}{s\,p}},& \mbox{ if } s\,p\le N,\\
q\ge 1,& \mbox{ if }s\,p>N.
\end{array}
\right.
\]
Then $u\in C^{\beta}_{\rm loc}(\Omega)$, where
\[
\beta=\frac{\alpha \gamma (p-1)}{\gamma (p-1)+(\alpha p+N)},
\]
with $\gamma$ as in \eqref{gamma} and $\alpha$ as in Lemma \ref{lm:transfer}. 
\par
More precisely, for every ball $B_{R_0}(z)\Subset\Omega$ we have the estimate
\[
\begin{split}
[u]_{C^{\beta}(B_{R_0}(z))}^p
&\le C\,\left[1+\|u\|_{L^\infty(B_{R_1}(z))}^p+(\mathrm{Tail}_{p-1,sp}(u;z,R_1))^{p}+ \|f\|^{p'}_{L^q(B_{R_1}(z))}\right],
\end{split}
\]
where 
\[
R_1=R_0+\frac{\mathrm{dist}(B_{R_0}(z),\partial\Omega)}{2}.
\] 
Here, the constant $C$ depends only $N,p,s,q,R_0$ and $\mathrm{dist}(B_{R_0}(z),\partial\Omega)$.
\end{Theorem}
\begin{proof} The proof is almost identical with the proof of Theorem 3.6 of \cite{BLS}. The only difference is that there is a parameter $\e$ and an additional term when applying Lemma \ref{lm:transfer}.  We take a ball $B_{R_0}(z)\Subset\Omega$ and set
\[
\mathrm{d}=\mathrm{dist}(B_{R_0}(z),\partial\Omega)>0\qquad \mbox{ and }\qquad R_1=\frac{\mathrm{d}}{2}+R_0.
\] 
Choose a point $x_0\in B_{R_0}(z)$ and consider the ball $B_{4R}(x_0)$ with $R<\min\{1,\mathrm{d}/8\}$. If\footnote{In the case when $s\,p=N$, the proof is similar, instead using Lemma \ref{lm:transfer} with an exponent $m>q'$.} $s\,p\not =N$, applying Lemma \ref{lm:transfer} and obtain 
\[
\begin{split}
&\fint_{B_r(x_0)} |u-\overline u_{x_0,r}|^p\,dx\\&\le C\Big\{\mathrm{Tail}_{p-1,sp}(u,x_0,4R)^p+\|u\|_{L^\infty(B_{4R}(x_0))}^p+\|f\|_{L^q(B_{4R}(x_0))}^{p'}+1\Big\}\Big\{\big(\frac{R}{r}\big)^N \e^\frac{p-2}{p-1}R^\gamma+\big(\frac{R}{r}\big)^N\e+\big(\frac{r}{R}\big)^{\alpha p}\Big\}
\end{split}
\] 
for every $0<r\le R<\min\{1,\mathrm{d}/8\}$.

As in the proof of Theorem 3.6 in \cite{BLS} it is straightforward to estimate these terms and obtain {
\[
\begin{split}
&\fint_{B_r(x_0)} |u-\overline u_{x_0,r}|^p\,dx\\&\le C\Big\{\mathrm{Tail}_{p-1,sp}(u,z,R_1)^p+\|u\|_{L^\infty(B_{R_1}(z))}^p+\|f\|_{L^q(B_{R_1}(z))}^{p'}+1\Big\}\Big\{\big(\frac{R}{r}\big)^N \e^\frac{p-2}{p-1}R^\gamma+\big(\frac{R}{r}\big)^N\e+\big(\frac{r}{R}\big)^{\alpha p}\Big\}
\end{split}
\] }
where $C=C(N,s,p,q)>0$. 

To simplify the notation, let 
$$
A=\mathrm{Tail}_{p-1,sp}(u,z,R_1)^p+\|u\|_{L^\infty(B_{R_1}(z))}^p+\|f\|_{L^q(B_{R_1}(z))}^{p'}+1.
$$
Then the above estimate reads
$$
\fint_{B_r(x_0)} |u-\overline u_{x_0,r}|^p\,dx\leq CA\left(\e^\frac{p-2}{p-1}\left(\frac{R}{r}\right)^{N}R^{\gamma}+\e \left(\frac{R}{r}\right)^{N}+\left(\frac{r}{R}\right)^{\alpha p}\right).
$$
We will now see that for a specific choice of $R$ and $\delta$ in terms of $r$, this implies that {$$\fint_{B_r(x_0)} |u-\overline u_{x_0,r}|^p\,dx$$} decays in a power fashion. Indeed, let {
$$
\e = \left(\frac{r}{R}\right)^{\alpha p+N},\quad R=r^{\sigma},
$$
where
$$
\sigma=\frac{\alpha p+N}{\alpha p+N+\gamma(p-1)}\in(0,1).
$$
Then
$$
\fint_{B_r(x_0)} |u-\overline u_{x_0,r}|^p\,dx\leq CA r^{\beta p},
$$
for $x_0\in B_{R_0}(z)$ and  $r<\min\{1,(\mathrm{d}/8)^\frac{1}{\sigma}\}$ where $\beta=\frac{\alpha \gamma (p-1)}{\gamma (p-1)+(\alpha p+N)}$. 
This shows that $u$ belongs to the Campanato space\footnote{We refer to \cite[Chapter 2]{Giusti} for the necessary details regarding the Campanato space $L^{q,\lambda}$.}} $\mathcal{L}^{p,N+\beta\,p}(B_{R_0}(z))$, which is isomorphic to $C^{\beta}(\overline{B_{R_0}(z)})$. The proof is complete.
\end{proof}

\subsection{Final H\"older regularity}
In order to prove Theorem \ref{teo:2}, we first establish the following stability result. 
\begin{Lemma}[Stability in $L^\infty$]
\label{lm:stab}
Let $1<p<2$, $0<s<1$. Suppose $\Omega\subset\mathbb{R}^N$ is an open and bounded set and
 $f\in L^q_{\rm loc}(\Omega)$ with 
\[
\left\{\begin{array}{lr}
q>\big({p^*_s}\big)',& \mbox{ if } s\,p< N,\\
q>1,& \mbox{ if } s\,p= N,\\
q\ge 1,& \mbox{ if }s\,p>N.
\end{array}
\right.
\]
Consider a local weak solution $u\in W^{s,p}_{\rm loc}(\Omega)\cap L^{p-1}_{s\,p}(\mathbb{R}^N)$ of the equation
\[
(-\Delta_p)^s u=f,\qquad \mbox{ in }\Omega.
\]
Let $B_{4}\Subset\Omega$ and assume that
$$
\|u\|_{L^\infty(B_2)}+\int_{\mathbb{R}^N\setminus B_2} \frac{|u(x)|^{p-1}}{|x|^{N+s\,p}}\, dx\leq M\qquad \mbox{ and }\qquad \|f\|_{L^q(B_2)}\leq \eta.
$$
Suppose that $h\in X^{s,p}_u({B_\frac32,B_{4}})$ weakly solves 
$$
\left\{\begin{array}{rcll}
(-\Delta_p)^s h &=& 0,& \mbox{ in }{B_\frac32},\\
h&=&u, &\mbox{ in }\mathbb{R}^N\setminus {B_\frac32}.
\end{array}
\right.
$$
Then there is $\tau_{M}(\eta)$ such that 
\begin{equation}
\label{stable}
\|u-h\|_{L^\infty(B_{\frac54})}\leq \tau_{M}(\eta)
\end{equation}
and $\tau_{M}(\eta)$ converges to $0$ as $\eta$ goes to $0$.
\end{Lemma}
\begin{proof}

The existence of a bound of the form \eqref{stable} is a consequence of the triangle inequality and the local $L^\infty$ estimate for the equation (Theorem 3.8 in \cite{BP}). We will now prove that $\tau_{M}(\eta)\to 0$ as $\eta\to 0$.
\par
We assume towards a contradiction that there exist two sequences $\{f_n\}_{n\in\mathbb{N}}\subset L^{q}(B_2)$ and $\{u_n\}_{n\in\mathbb{N}}$ such that
\[
\|u_n\|_{L^\infty(B_2)}+\int_{\mathbb{R}^N\setminus B_2} \frac{|u_n|^{p-1}}{|x|^{N+s\,p}}\, dx\leq M,\qquad \|f_n\|_{L^q(B_2)}\to 0,
\]
but
\[
\liminf_{n\to\infty} \|u_n-h_n\|_{L^\infty(B_{\frac54})}>0.
\]
We note that by Lemma \ref{lm:cacc1}, any $u$ satisfying the assumptions of the lemma also satisfies the bound
\begin{equation*}
\label{eq:wspbound}
[u]_{W^{s,p}(B_\frac53)}\leq C(M,N,s,p).
\end{equation*}
Therefore, \eqref{prima} implies that for every $\e\in (0,1/2)$, we have
\[
\limsup_{n\to\infty}[u_n-h_n]^p_{W^{s,p}({\mathbb{R}^N})}\le C\e^{\frac{p-2}{p-1}}\limsup_{n\to\infty}\,\left(\int_{B_{\frac32}} |f_n|^{q}\,dx\right)^\frac{p'}{q}
+\e\lim_{n\to\infty}[u_n]^p_{W^{s,p}(B_\frac53)}\leq  C\e,
\]
where $C=C(M,N,p,s)>0$ is a constant. Since this holds for any $\e\in (0,1/2)$, we conclude that 
\begin{equation}
\label{azzero}
\lim_{n\to\infty}[u_n-h_n]^p_{W^{s,p}({\mathbb{R}^N})}=0.
\end{equation}
This, together with the fractional Sobolev inequality and Theorem 1.1 in \cite{DKP} implies that $h_n$ is locally uniformly bounded in $B_{3/2}$.
Theorem 3.1 in \cite{BLS} or Theorem \ref{teo:1} implies that $h_n$ is uniformly bounded in $C^{\beta}(B_{5/4})$ and Theorem~\ref{teo:holderf} implies that $u_n$ is uniformly bounded in $C^{\beta}(B_{5/4})$ for some $\beta>0$. Therefore, by the Ascoli-Arzel\`a theorem, we may conclude that $u_n-h_n$ converges uniformly in $\overline{B_{5/4}}$, up to a subsequence. By \eqref{azzero} we get that
\[
\lim_{n\to\infty} \|u_n-h_n\|_{L^\infty(B_{5/4})}=0,
\]
which gives the desired contradiction.
\end{proof}
The following proposition is a rescaled version of Theorem \ref{teo:2}.
\begin{prop} 
\label{prop:caffsilv}
Let $1<p<2$, $0<s<1$. Take $q$ such that
\[
\left\{\begin{array}{lr}
q>\big({p^*_s}\big)',& \mbox{ if } s\,p< N,\\
q>1,& \mbox{ if } s\,p= N,\\
q\ge 1,& \mbox{ if }s\,p>N,
\end{array}
\right.
 \]
and define
$$
\Theta = \min\Big(1,\frac{sp-N/q}{p-1}\Big).
$$
For every $0<\varepsilon<\Theta$ there exists $\eta=$ $\eta(N,p,q,s,\varepsilon)>0$ such that if $f\in L^q_{\rm loc}(B_4(x_0))$ and
\[
\|f\|_{L^q({B_2(x_0))}}\leq \eta,
\] 
then every local weak solution $u\in W^{s,p}_{\rm loc}(B_4(x_0))\cap L^{p-1}_{s\,p}(\mathbb{R}^N)$ of the equation
\[
(-\Delta_p)^s u=f,\qquad \mbox{ in }B_4(x_0),
\]
such that
\begin{equation}
\label{startup}
\|u\|_{L^\infty({B_2(x_0))}}\leq 1,\qquad \int_{\mathbb{R}^N\setminus {B_2(x_0)}}\frac{|u|^{p-1}}{|x|^{N+s\,p}}\, dx\leq 1
\end{equation}
belongs to $C^{\Theta-\varepsilon}(\overline{B_{1/8}(x_0)})$ with the estimate
$$
[u]_{C^{\Theta-\varepsilon}(\overline{B_{1/8}}(x_0))}\leq C(N,p,q,s,\varepsilon),
$$
for some constant $C(N,p,q,s,\varepsilon)>0$.
\end{prop}
\begin{proof} 
Without loss of generality, we may assume that $x_0=0$. We divide the proof in two parts.
\vskip.2cm\noindent
{\bf Part 1: Regularity at the origin}. We claim that for any $0<\varepsilon<\Theta$ and every $0<r<1/2$, there exists $\eta=\eta(N,p,q,s,\varepsilon)>0$ and a constant $C=C(N,p,q,s,\varepsilon)>0$ such that if $f$ and $u$ are as above, then we have
\[
\sup_{x\in B_r} |u(x)-u(0)|\leq C\,r^{\Theta-\varepsilon}.
\]
Without loss of generality, we assume $u(0)=0$. Let us fix $0<\varepsilon<\Theta$. Then we remark that it is enough to prove that there exists $\lambda<1/2$ and $\eta>0$ (depending on $N,p,q,s$ and $\varepsilon$) such that if $f$ and $u$ are as above, then
\begin{equation}
\label{eq:keq}
\sup_{B_{{2\lambda^k}}}|u|\leq \lambda^{k\,(\Theta-\varepsilon)},\qquad \int_{\mathbb{R}^N\setminus {B_2}}\left|\frac{u(\lambda^k\, x)}{\lambda^{k\,(\Theta-\varepsilon)}}\right|^{p-1}\,|x|^{-N-s\,p}\, dx\leq 1,
\end{equation}
for every $k\in\mathbb{N}$. Indeed, if this is true, then for every $0<r<1/2$, there exists $k\in \mathbb{N}$ such that {$2\lambda^{k+1}< r\le 2\lambda^k$}. Using the first property from \eqref{eq:keq}, we deduce that
\[
\sup_{B_r} |u|\le \sup_{B_{2\lambda^k}} |u|\le \lambda^{k\,(\Theta-\varepsilon)}=\frac{1}{\lambda^{\Theta-\varepsilon}}\,\lambda^{(k+1)\,(\Theta-\varepsilon)}\le C\,r^{\Theta-\varepsilon},
\]
where $C=\frac{1}{(2\lambda)^{\Theta-\epsilon}}$ as desired.
\par
We prove \eqref{eq:keq} by an induction argument. First, we note that \eqref{eq:keq} holds true for $k=0$,  using the assumptions in \eqref{startup}. 
Suppose \eqref{eq:keq} is valid up to $k$. We prove that this is also valid for $k+1$ assuming that
\[
\|f\|_{L^q({B_2})}\le \eta,
\]
for small enough $\eta$, which is independent of $k$.
We define
$$
w_k(x)=\frac{u(\lambda^k x)}{\lambda^{k\,(\Theta-\varepsilon)}}.
$$
We observe that by the hypotheses, it follows that
\begin{equation}\label{eq:wkass}
\|w_k\|_{L^\infty({B_2})}\leq 1 \qquad \mbox{ and } \qquad \int_{\mathbb{R}^N\setminus {B_2}}\frac{|w_k|^{p-1}}{|x|^{N+s\,p}}\,dx\leq 1.
\end{equation}
Furthermore,
$$
(-\Delta_p)^s w_k (x) = \lambda^{k\,[sp\,-(\Theta-\varepsilon)(p-1)]}\,f(\lambda^k\, x)=:f_k(x).
$$
We notice that 
$$\
\|f_k\|_{L^{q}({B_2})}=\lambda^{k(sp\,-(\Theta-\varepsilon)(p-1))}\lambda^{-\frac{N}{q}\,k}\,\left(\int_{{B_{2\lambda^k}}}|f|^{q}\,dx\right)^{\frac{1}{q}}\leq \|f\|_{L^{q}({B_2})}\le \eta,
$$
where we have used the hypotheses on $f$, the definition of $\Theta$, and again the fact that $\lambda<1/2$. 
By Proposition 2.12 in \cite{BLS}, we consider {$h_k\in X_{w_k}^{s,p}(B_\frac32,B_4)$} to be the weak solution of 
$$
\left\{\begin{array}{rcll}
(-\Delta_p)^s h &=& 0,& \mbox{ in } B_\frac32,\\
h&=&w_k,& \mbox{ in } \mathbb{R}^N\setminus B_\frac32.
\end{array}
\right.
$$
From Lemma \ref{lm:stab}, we obtain 
\[
 \|w_k-h_k\|_{L^\infty(B_{\frac54})}<\tau_\eta,
\]
{where $\tau_\eta\to 0$ as $\eta\to 0$ and $\tau_\eta$ is independent of $k$.}
Therefore,
\begin{equation}
\label{estimata}
\begin{split}
|w_k(x)|&\leq |w_k(x)-h_k(x)|+|h_k(x)-h_k(0)|+|h_k(0)-w_k(0)|\\
&\leq  2\tau_\eta+[h_k]_{C^{\Theta-\varepsilon/2}(B_{1})}\,|x|^{\Theta-\frac{\varepsilon}{2}},\qquad\qquad \mbox{ for }x\in B_{1}.
\end{split}
\end{equation}
To obtain the above estimate, we have also used the fact that $h_k$ belongs to $C^{\Theta-\varepsilon/2}(\overline{B_{1}})$,  which follows from Theorem \ref{teo:1} and Remark \ref{rem:covering}, with the estimate
\[
[h_k]_{C^{{\Theta-\varepsilon/2}}(B_{1})}\le C\, \left(\|h_k\|_{L^\infty({B_{\frac{5}{4}}})}+\Tail_{p-1,sp}\big(h_k;0,{\frac{5}{4}}\big)\right)\leq C_1,\quad C_1=C_1(N,p,q,s,\varepsilon).
\]
We obtained the above estimate by observing that the quantities in the right-hand side are uniformly bounded, independently of $k$. To this end, Lemma \ref{lm:stab} and \eqref{eq:wkass} along with the triangle inequality gives that
\[
\|h_k\|_{L^\infty(B_{\frac{5}{4}})}\le  \|h_k-w_k\|_{L^\infty(B_{\frac{5}{4}})}+\|w_k\|_{L^\infty(B_{\frac{5}{4}})}\le \tau_\eta+1.
\]
\par
For the tail term, by the triangle inequality, the hypothesis on $w_k$ and \eqref{seconda} combined with Lemma \ref{lm:cacc1}, we obtain
\begin{equation*}
\begin{split}
\Tail_{p-1,sp}\big(h_k;0,{\frac{5}{4}}\big)&\leq C\Big(\int_{\mathbb{R}^N\setminus B_{\frac{5}{4}}}\frac{|h_k-w_k|^{p-1}}{|x|^{N+ps}}\,dx\Big)^\frac{1}{p-1}+C\Big(\int_{\mathbb{R}^N\setminus B_{\frac{5}{4}} }\frac{|w_k|^{p-1}}{|x|^{N+ps}}\,dx\Big)^\frac{1}{p-1}\\
&\leq C\Big(\int_{B_\frac{3}{2}\setminus B_{\frac{5}{4}}}\frac{|h_k-w_k|^{p-1}}{|x|^{N+ps}}\,dx\Big)^\frac{1}{p-1}+\Big(\int_{\mathbb{R}^N\setminus B_{2}}\frac{|w_k|^{p-1}}{|x|^{N+ps}}\,dx\Big)^\frac{1}{p-1}\\
&\quad+\Big(\int_{B_2\setminus B_{\frac{5}{4}}}\frac{|w_k|^{p-1}}{|x|^{N+ps}}\,dx\Big)^\frac{1}{p-1}\\
&\leq  C(1+\eta),
\end{split}
\end{equation*} 
with $C=C(N,s,p,q)$. We also made use of \eqref{eq:wkass} and that $h_k=w_k$ outside $B_{3/2}$, by construction. Therefore, the estimate \eqref{estimata} is uniform in $k$.
Let
$$
w_{k+1}(x)=\frac{u(\lambda^{k+1}\, x)}{\lambda^{(k+1)\,(\Theta-\varepsilon)}}=\frac{w_k(\lambda\, x)}{\lambda^{\Theta-\varepsilon}}.
$$
We can transfer estimate \eqref{estimata} to $w_{k+1}$ by choosing $\eta$ so that {$2\tau_\eta<\lambda^\Theta$} and $\lambda$ small enough. Indeed, we observe that
\[
\begin{split}
|w_{k+1}(x)|\leq  2\tau_\eta\,\lambda^{\varepsilon-\Theta}+C_1\,\lambda^{\varepsilon/2}|x|^{\Theta-\varepsilon/2}\leq (1+C_1\,|x|^{\Theta-\varepsilon/2})\,\lambda^{\varepsilon/2},\qquad  x\in B_\frac{1}{\lambda}.
\end{split}
\] 
In particular, the above estimate gives that $\|w_{k+1}\|_{L^\infty({B_2})}\leq 1$ for $\lambda$ satisfying
\begin{equation}
\label{uno}
\lambda<\min\left\{{\frac{1}{4}},(1+C_12^{\Theta-\e/2})^{-\frac{2}{\varepsilon}}\right\}.
\end{equation}
This information, rescaled back to $u$, gives precisely the first part of \eqref{eq:keq} for $k+1$. To obtain the second part of \eqref{eq:keq}, we use the upper bound for $|w_{k+1}|$ and the fact that $\Theta<\frac{sp}{p-1}$, which gives
\begin{equation}
\label{nonlocal1}
\begin{split}
\int_{B_{\frac{1}{\lambda}}\setminus B_{2}} \frac{|w_{k+1}|^{p-1}}{|x|^{N+s\,p\,}}\,dx&\leq \lambda^{\varepsilon\, (p-1)/2} \int_{B_{\frac{1}{\lambda}}\setminus B_{2}}  \frac{(1+C_1\,|x|^{\Theta-\varepsilon/2})^{p-1}}{|x|^{N+s\,p}}\,dx\\
&\leq (1+C_1)^{p-1}\,\lambda^{\varepsilon\, (p-1)/2} \int_{B_{\frac{1}{\lambda}}\setminus B_{2}}  \frac{1}{|x|^{N+sp+(\varepsilon/2-\Theta)\,(p-1)}}\,dx\\
&\leq \frac{C_2}{s\,p-(\Theta-\varepsilon/2)\,(p-1)}\,\lambda^{\varepsilon\,(p-1)/2}.
\end{split}
\end{equation}
 Since $|w_k|\leq 1$ in $B_2$, a change of variable gives
\begin{equation}
\label{nonlocal2}
\int_{B_{\frac{2}{\lambda}}\setminus B_{\frac{1}{\lambda}}} \frac{|w_{k+1}|^{p-1}}{|x|^{N+s\,p\,}}\,dx=\lambda^{(\varepsilon-\Theta)\,(p-1)+s\,p}\,\int_{B_2\setminus B_1} \frac{|w_k(x)|^{p-1}}{|x|^{N+s\,p}}\,dx\leq {C_3\,\lambda^{\varepsilon\,(p-1)/2}}.
\end{equation}
In addition, by the integral bound on $w_k$ in \eqref{eq:wkass} and using that $\text{Tail}_{p-1,sp}(w_k;0,2)\leq 1$, we get
\begin{equation}
\label{nonlocal3}
\int_{\mathbb{R}^N\setminus B_{\frac{2}{\lambda}}} \frac{|w_{k+1}(x)|^{p-1}}{|x|^{N+s\,p}}\,dx= \lambda^{(\varepsilon-\Theta)\,(p-1)+s\,p}\,\int_{\mathbb{R}^N\setminus B_2}\frac{|w_k(x)|^{p-1}}{|x|^{N+s\,p}}\,dx\leq \lambda^{\varepsilon\, (p-1)/2}.
\end{equation}
The condition $\lambda<1/2$ and the fact that
\[
(\varepsilon-\Theta)\,(p-1)+s\,p\ge \varepsilon\,\frac{p-1}{2}
\]
is used in both estimates.
Here the constants $C_2$ and $C_3$ depend on $N,p,q,s$ and $\varepsilon$ only.
By \eqref{nonlocal1}, \eqref{nonlocal2} and \eqref{nonlocal3},
we get that the second part of \eqref{eq:keq} holds, provided that
$$
\left(\frac{C_2}{\varepsilon\,(p-1)}+C_3+1\right)\,\lambda^{\varepsilon\,(p-1)/2}\leq 1.
$$
Recalling \eqref{uno}, we finally obtain that \eqref{eq:keq} holds true at step $k+1$ as well, when $\lambda$ and $\eta$ (depending on $N,p,q,s$ and $\varepsilon$) are chosen so that
\[
\lambda<\min\left\{\frac{1}{2},( 1+C_12^{\Theta-\e/2})^{-\frac{2}{\varepsilon}}, \left(\frac{C_2}{\varepsilon\,(p-1)}+C_3+1\right)^\frac{2}{\varepsilon\,(p-1)}\right\}\qquad \mbox{ and }\qquad \tau_\eta<\frac{\lambda^\Theta}{2}.
\]
The induction is complete.
\vskip.2cm\noindent
\textbf{Part 2:} We prove the desired regularity in the whole ball $B_{1/8}$. To this end, we take $0<\varepsilon<\Theta$ and choose the associated $\eta$, obtained in {\bf Part 1}. Take $z_0\in B_{1}$, let $L=2^{N+1}\,(1+|B_2|)$ and define
$$
v(x):=L^{-\frac{1}{p-1}}\,u\left(\frac{x}{2}+z_0\right),\qquad x\in \mathbb{R}^N.
$$
We observe that $v\in W^{s,p}_{\rm loc}(B_4)\cap L^{p-1}_{s\,p}(\mathbb{R}^N)$ and that $v$ is a weak solution in $B_4$ of 
\[
(-\Delta_p)^s v(x)=\frac{2^{-sp}}{L}\,f\left(\frac{x}{2}+z_0\right)=:\widetilde f(x),
\]
with
\[
\left\|\widetilde f\right\|_{L^{q}(B_2)}=\frac{2^{N/q-sp}}{L}\,\|f\|_{L^{q}(B_{1}(z_0))}\le \frac{2^{N/q-sp}}{L}\,\eta<\eta.
\]
Moreover, by construction, we have 
\[
\|v\|_{L^\infty(B_2)}\leq 1.
\] 
Observing that $B_{1}(z_0)\subset B_2$ along with the definition of $L$ and the hypotheses in \eqref{startup}, we get
\[
\begin{split}
\int_{\mathbb{R}^N\setminus B_2}\frac{|v(x)|^{p-1}}{|x|^{N+s\,p}}\, dx&=\frac{2^{-s\,p}}{L}\,\int_{\mathbb{R}^N\setminus B_{1}(z_0)}\frac{|u(y)|^{p-1}}{|y-z_0|^{N+s\,p}}\,dy\\
&\le \frac{1}{L}\,\left(\frac{1}{2}\right)^{s\,p}\,\left(\frac{2}{2-|z_0|}\right)^{N+s\,p}\,\int_{\mathbb{R}^N\setminus B_2}\frac{|u(y)|^{p-1}}{|y|^{N+s\,p}}\,dy+\frac{2^{N}}{L}\,\|u\|^{p-1}_{L^{p-1}(B_2)}\\
&\leq \frac{2^{N}}{L}\,\int_{\mathbb{R}^N\setminus B_2}\frac{|u(y)|^{p-1}}{|y|^{N+s\,p}}dy+\frac{2^N\,|B_2|}{L}\,\|u\|^{p-1}_{L^\infty(B_2)}\leq 1.
\end{split}
\]
In the above estimate, we have also used Lemma 2.3 in \cite{BLS} with the balls $B_{1}(z_0)\subset B_2$. Therefore applying {\bf Part 1} to $v$, we obtain
$$
\sup_{x\in B_r}|v(x)-v(0)|\leq C\,r^{\Theta-\varepsilon},\quad 0<r<\frac{1}{2},
$$
which in terms of $u$ is same as
\begin{equation}
\label{supest}
\sup_{x\in B_r(z_0)}|u(x)-u(z_0)|\leq C\,L^\frac{1}{p-1}\,r^{\Theta-\varepsilon},\qquad 0<r<\frac{1}{4}.
\end{equation}
We remark that the above estimate holds for any $z_0\in B_{1}$. We choose any pair $x,y\in B_{1/8}$ such that $|x-y|= r$. Then $r<1/4$. Setting $z=(x+y)/2$, we apply \eqref{supest} with $z_0=z$ and obtain
\[
\begin{split}
|u(x)-u(y)|\leq |u(x)-u(z)|+|u(y)-u(z)|&\leq 2\sup_{w\in B_r(z)}|u(w)-u(z)|\\
&\leq 2\,C\,L^\frac{1}{p-1}\,r^{\Theta-\varepsilon}=2\,C\,L^\frac{1}{p-1}\,|x-y|^{\Theta-\varepsilon},
\end{split}
\]
which is the desired result.
\end{proof}

We are now ready to give the proof of the final H\"older regularity result.

\begin{proof}[~Proof of Theorem \ref{teo:2}]
Without loss of generality, we may assume $x_0=0$. We modify $u$ in such a way that it fits into the setting of Proposition \ref{prop:caffsilv}. Let
\[
\mathcal{A}_R=\|u\|_{L^\infty(B_{2R})}+\left(R^{s\,p}\,\int_{\mathbb{R}^N\setminus B_{2R}}\frac{|u(y)|^{p-1}}{|y|^{N+s\,p}}\,  dy\right)^\frac{1}{p-1}+\left(\frac{R^{sp-N/q}\|f\|_{L^{q}(B_{2R})}}{\eta}\right)^\frac{1}{p-1},\]
where we have chosen $\e\in (0,\Theta)$ and $\eta$ as in Proposition \ref{prop:caffsilv}.
Note that $u$ is locally bounded by Theorem 3.8 in \cite{BP}. By scaling arguments, it is enough to prove that the rescaled function
\[
u_R(x):=\frac{1}{\mathcal{A}_R}\,u(R\,x),\qquad \mbox{ for }x\in B_4,
\]
satisfies the estimate
\[
[u_R]_{C^{\Theta-\e}(B_{1/8})}\leq C.
\]
It is straightforward to see that the choice of $\mathcal{A}_R$ implies
$$
\|u_R\|_{L^\infty(B_2)}\leq 1,\qquad \int_{\mathbb{R}^N\setminus B_2}\frac{|u_R|^{p-1}}{|x|^{N+s\,p}}\, dx\leq 1. 
$$
Also, $u_R$ is a local weak solution of
$$
(-\Delta_p)^s u_R\, (x) = \frac{R^{sp}}{\mathcal{A}_R^{p-1}}\,f(R\,x):= f_R(x),\qquad x\in B_4,
$$
with $\|f_R\|_{L^{q}(B_{2})}\le \eta$. Therefore, applying Proposition \ref{prop:caffsilv} to $u_R$, we obtain
\[
[u_R]_{C^{\Theta-\e}(B_{1/8})}\leq C.
\]
After scaling back, this concludes the proof.
\end{proof}

\appendix

\section{Useful Inequalities}
The following inequality follows from from \cite[(I), Page 95]{PL}}.
\begin{Lemma}\label{PLine}
Let $1<p<\infty$ and $a,b,c,d\in\mathbb{R}$. Then we have
\begin{equation*}\label{neqn3}
(J_{\beta+1}(a)-J_{\beta+1}(b))(a-b)\geq {2^{-1}}(|a|^{\beta-1}+|b|^{\beta-1})|a-b|^2.
\end{equation*}
\end{Lemma}

\begin{Lemma}\label{lem:ineq}
Let {$p\geq 2$ and $q>1$}. For every $a,b\in\mathbb{R}$, we have
\begin{equation*}
J_{q}(a-b)\,\Big(J_p(a)-J_p(b)\Big)\ge (p-1)\,\left(\frac{q}{p-2+q}\right)^q\, \left||a|^\frac{p-2}{q} a-|b|^\frac{p-2}{q}{b}\right|^q.
\end{equation*}
\end{Lemma}
This is Lemma A.1 in \cite{BLS}.

\begin{Lemma} \label{lemma:sing_ineq_1}
Let $1<p<2$ and $a,c \in \R^n$. Then
\[
(p-1)\frac{|a-c|^2}{(|a|+|c|)^{2-p}} \leq \left(|a|^{p-2}a-|c|^{p-2}c\right)\cdot (a-c).
\]
\end{Lemma}
This is inequality $(2.3)$ in \cite{JL09}.

\begin{Lemma} \label{lemma:sing_ineq_2}
Let $\gamma\geq 2$, $1<p<2$ and $a,b,c,d \in \R$. Then
\[
\begin{split}
    \big(J_\gamma(a-b)&-J_\gamma(c-d)\big)\big(J_p(a-c)-J_p(b-d)\big) \\
    &\geq 4\frac{(\gamma-1)(p-1)}{\gamma^2}\left||a-b|^{\frac{\gamma-2}{2}}(a-b)-|c-d|^{\frac{\gamma-2}{2}}(c-d)\right|^2\big(|a-c|+|b-d|\big)^{p-2}.
\end{split}
\]
\end{Lemma}
\begin{proof}
The proof is just a combination of Lemma \ref{lem:ineq} and Lemma \ref{lemma:sing_ineq_1}.
\end{proof}

The following inequality can be found in \cite[Lemma A.3]{BLS}.
\begin{Lemma}\label{lemma:holder}
Let $\gamma\ge 1$. Then for every $A,B\in\mathbb{R}$, we have
\begin{equation*}
\label{holder}
\left||A|^{\gamma-1}\,A-|B|^{\gamma-1}\,B\right|\ge \frac{1}{C}\, |A-B|^\gamma,
\end{equation*}
for some constant $C=C(\gamma)>0$.
\end{Lemma}

\noindent {\textsf{Prashanta Garain\\Department of Mathematical Sciences\\
Indian Institute of Science Education and Research Berhampur\\
Berhampur, Odisha 760010, India}\\ 
\textsf{e-mail}: pgarain92@gmail.com\\

\noindent {\textsf{Erik Lindgren\\  Department of Mathematics\\ KTH -- Royal Institute of Technology\\ 100 44, Stockholm, Sweden}  \\
\textsf{e-mail}: eriklin@math.kth.se\\

\end{document}